%% file: IEEE_paper_arXiv.tex
\documentclass[lettersize,journal]{IEEEtran}
\usepackage{amsmath,amsfonts}
\usepackage{mathptmx}
\usepackage{amssymb}
\usepackage{amsthm}
\usepackage{algorithmic}
\usepackage{algorithm}
\usepackage{array}
\usepackage[caption=false,font=normalsize]{subfig}
\usepackage{textcomp}
\usepackage{stfloats}
\usepackage{url}
\usepackage{verbatim}
\usepackage{graphicx}
\usepackage[noadjust]{cite}
\usepackage{bbm}
\usepackage{enumitem}
\usepackage{multirow}
\usepackage{makecell}
\usepackage{caption}
\usepackage{geometry}
\usepackage{rotating}
\usepackage{pdflscape}
\usepackage{lscape}
\usepackage{longtable}
\usepackage[table,xcdraw]{xcolor}
\usepackage{tabularx}
\usepackage{hyperref}
\usepackage[numbers]{natbib}

\usepackage{todonotes}

\newcommand{\lowerh}{\underline{h}_A}
\newcommand{\upperh}{\overline{h}_A}

\newcommand{\conv}{\text{conv}}

\newcommand{\kf}{_{K_1}}
\newcommand{\ks}{_{K_2}}
\newcommand{\jf}{_{J_1}}
\newcommand{\js}{_{J_2}}
\newcommand{\diag}{\text{diag}}
\newcommand{\NP}{${NP}$}
\renewcommand{\And}{\&}

\newcommand{\mybar}[1]{\ifmmode\setbox0\hbox{$#1$}%
\else
\setbox0\hbox{#1}%
\fi
\makebox[\the\wd0][c]{%
\rule[0.42\ht0]{0.75\wd0}{0.7pt}}\hspace*{-\the\wd0}#1}

\newcommand{\myZ}{\mathbf{Z}}
\newcommand{\myY}{\mathbf{Y}}
\newcommand{\myX}{\mathbf{X}}

\newcommand{\bbR}{\mathbb{R}}

\newcommand{\bm}[1]{\boldsymbol{#1}}

\newcommand\ignore[1]{}

\makeatletter
\newcommand{\pushright}[1]{\ifmeasuring@#1\else\omit\hfill$\displaystyle#1$\fi\ignorespaces}
\newcommand{\pushleft}[1]{\ifmeasuring@#1\else\omit$\displaystyle#1$\hfill\fi\ignorespaces}
\makeatother

\theoremstyle{definition}

\input def_arXiv.tex

\begin{document}

\title{Strong Formulations for \\ Hybrid System Control}

\author{Jisun Lee, Hyungki Im, Alper Atamt\"{u}rk
\thanks{November 24, 2024. This work was supported, in part, by Grant 12951270 from the Office of Naval Research of the US DoD.}
\thanks{All authors are with the Department of Industrial Engineering and Operations Research, University of California, Berkeley, CA 94720. (e-mails: jisun\_lee@berkeley.edu, hyungki.im@berkeley.edu, atamturk@berkeley.edu)}}




\maketitle

\begin{abstract}
We study the mixed-integer quadratic programming formulation of an $n$-period hybrid control problem with a convex quadratic cost function and linear dynamics. 
We first give the convex hull description of the single-period, two-mode problem in the original variable space through two new classes of valid cuts.
These cuts are then generalized to the single-period, multi-mode, multi-dimensional case and applied to solve the general $n$-period hybrid control problem.
Computational experiments demonstrate the effectiveness of the proposed strong formulations derived through the cut generation process in the original variable space. These formulations yield a substantial reduction in computational effort for synthetic test instances and instances from the energy management problem of a power-split hybrid electric vehicle.
\end{abstract}

\begin{IEEEkeywords}
Hybrid system control, feasibility cuts, nonlinear cuts, disjunctive programming, projection.
\end{IEEEkeywords}

\section{Introduction} \label{sec:Intro}

\IEEEPARstart{A} hybrid control system is one with a mixture of discrete and continuous dynamics \citep{borrelli2017predictive}. 
Controlling a hybrid system for a long-term or infinite time horizon is computationally challenging. Moreover, there may be external disturbances that cannot be predicted in advance. Therefore, the model predictive control (MPC) approach tackles the challenge of uncertainty over long horizons by solving short-term problems iteratively. 
In each period $t$, the current state is measured, and an $n$-period hybrid control problem (HCP) is solved to determine the control actions for $[t,t+n)$. The process is repeated over a rolling horizon after implementing the control only for time $t$. Although the MPC approach performs well in practice, it requires real-time computations and considerations for stability, robustness, and feasibility of the solutions.

\IEEEpubidadjcol
In this paper, we study the $n$-period HCP subproblem solved at each iteration of the MPC approach. Our primary focus is to
provide strong convex relaxations of the problem through novel valid inequalities in the original space of variables, leading to improved lower bounds and, in turn, reducing the search effort required by the branch-and-bound (B\&B) algorithms. 

The $n$-period HCP can be formulated as the following mixed-integer quadratic program (MIQP):
\begin{equation}
    \begin{aligned}
        \min \ &\sum_{t=1}^n \left( \bm{x}_t^{\top} \bm{Q}_t \bm{x}_t + \bm{y}_t^{\top} \bm{R}_t \bm{y}_t + \bm{z}_t^\top \bm{S}_t \bm{z}_t \right) \hspace{-1.2mm}&& \hspace{-2.7mm}+ \bm{x}_{n+1}^{\top} \bm{Q}_{n+1} \bm{x}_{n+1} \hspace{-3mm}\\
        \text{ s.t.} \ & \bm{x}_{t+1} = \bm{A}_t \bm{x}_t + \bm{B}_t \bm{y}_t + \bm{C}_t \bm{z}_t + \bm{f}_t, && t \in [n] \\  
        & \mathbbm{1}^\top \bm{z}_t \leq 1, && t \in [n]\\
        & \bm{G}_{t} \bm{z}_t \leq \bm{y}_t \leq \bm{H}_{t} \bm{z}_t, && t \in [n]\\
        & \bm{\ell}_{t} \leq \bm{x}_t \leq \bm{u}_{t}, && t \in [n+1]\\
        & \bm{x}_t \in \mathbb{R}^{d_x}, && t \in [n+1]\\
        & \bm{y}_t \in \mathbb{R}^{d_y}, && t \in [n]\\
        & \bm{z}_t \in \{0,1\}^{d_z}, && t \in [n],
    \end{aligned}
    \label{ORIG_PROB}
\end{equation} 
\noindent where $\bm{x}_t$ is the state variable, $\bm{y}_{t}$ is the control variable, and $\bm{z}_t$ is the indicator of the system mode at time $t$. The cost matrices $\bm{Q}_t \in \mathbb{R}^{d_x \times d_x}$, $\bm{R}_t \in \mathbb{R}^{d_y \times d_y}$, $\bm{S}_t \in \mathbb{R}^{d_z \times d_z}$ are symmetric positive semidefinite matrices, typically diagonal.
The $j$-th columns of $\bm{G}_{t}$ and $\bm{H}_{t}$ correspond to the lower and upper bounds of $\bm{y}_{t}$ when the $j$-th mode is chosen at time $t$. Vectors $\bm{\ell}_{t}$ and $\bm{u}_{t}$ represent the lower and upper bounds of $\bm{x}_t$.  There are no restrictions on the dynamics matrices $\bm{A}_t \in \mathbb{R}^{d_x \times d_x}$, $\bm{B}_t \in \mathbb{R}^{d_x \times d_y}$, $\bm{C}_t \in \mathbb{R}^{d_x \times d_z}$ and $\bm{f}_t \in \mathbb{R}^{d_x}$.
For simplicity of notation, we assume $\bm{f}_t = \bm{0}$, but demonstrate that our results extend to any $\bm{f}_t \in \mathbb{R}^{d_x}$.

The singe-period HCP is polynomial-time solvable as 
it requires solving $d_z+1$ convex quadratic programs
each corresponding to a fixed value of $\bm{z} \in \{\bm{0}, \bm{e}_1,\ldots, \bm{e}_{d_z}\}$. Conversely, the multi-period HCP is \NP-hard even with $d_z=1$. 
Consider the capacitated lot-sizing problem
\begin{equation*}
    \begin{aligned}
        \min \ & \sum_{t=1}^{n+1} q_t( x_t) + \sum_{t=1}^{n} r_t (y_t) + \sum_{t=1}^n s_t (z_t) \\
        \text{s.t. }\ & x_1 = f_0 \\
        & x_{t+1} = x_t + y_t - f_t, &&  t \in [n] \\
        & \ell_t z_t \leq y_t \leq u_t z_t, && t \in [n] \\
        & x_{t+1}, y_t \in \mathbb{R}, \ z_t \in \{0,1\}, && t \in [n]
    \end{aligned}
\end{equation*}
where $x_t$ is the inventory, $y_t$ is the production amount, $z_t$ is the indicator of production, and $f_t$ is the demand in period $t$. The cost consists of the holding cost $q_t$, the production cost $r_t$, and the setup cost $s_t$ in period $t$. 
The capacitated lot-sizing problem is known to be \NP-hard \citep{cls-nphard} and is a special case of \eqref{ORIG_PROB} with $d_x, d_y, d_z = 1$.

{
\textit{Contributions.} 
In this paper, we present strong formulations for \eqref{ORIG_PROB} in the original space of the variables through valid cuts. This approach contrasts with alternative disjunctive programming formulations of the problem given in an extended space by introducing auxiliary variables to strengthen the convex relaxations. The advantage of the formulations in the original space is that they tend to scale better by avoiding a large number of auxiliary variables and judicious use of valid cuts. 

We start the analysis by giving the convex hull description for the simplest single-period one-dimensional case through two classes of valid cuts. The cut generation process for this simple case provides the basis for the more general problem. Subsequently, we generalize the cuts in phases to multi-period HCP with multi-dimensional state, control, and indicator variables. 

}

The rest of the paper is organized as follows. In Section \ref{sec:literature}, we review the literature on solving MIQPs arising in the hybrid control domain. Section \ref{sec:strong_form} is the main part of the paper, where we give a strong formulation for \eqref{ORIG_PROB} in the original variable space through cutting planes, starting with the simplest single-period one-dimensional case and then generalizing the results, in phases, to \eqref{ORIG_PROB}. 
In Section \ref{sec:experiment}, we present a computational study to test the effectiveness of the cut generation approach using synthetic data. In Section \ref{sec:HEV}, we apply our cutting-plane approach to an energy management problem of a power-split hybrid electric vehicle and illustrate its valuable impact on computations. In Section \ref{sec:conclusion}, we conclude with a few final remarks.


\section{Literature review} \label{sec:literature}

Solving finite-horizon HCPs within short sampling times has been a key challenge to controlling hybrid systems in real-time. Reducing the solution times of HCP optimization problems has been a major research thread to tackle this challenge.

Parametric programming has been extensively utilized in the MPC literature due to the need to solve similar optimization problems repeatedly in an iterative manner.
The closed-form solution for discrete-time linear quadratic optimal control problems involving continuous variables is a piecewise linear and continuous function \cite{bemporad2002explicit}. This result is extended to hybrid systems, demonstrating that the optimal control law for a finite-time hybrid system is a time-varying piecewise affine function in \cite{borrelli2005dynamic}.
 To compute the optimal control law offline, they formulate mixed-integer multiparametric programming problems and give a dynamic programming algorithm to derive solutions optimally across a set of states. Utilizing the piecewise affinity of the optimal control law, critical regions of the parameter space for optimality are driven in \cite{dua2002multiparametric}. The Multi-Parametric Toolbox (MPT) \citep{kvasnica2004multi} provides access to such precomputed control solutions for constrained linear systems. The efficiency of B\&B procedures for solving parametric MIQPs is analyzed, establishing worst-case bounds on the number of nodes explored in \cite{axehill2010improved}. Online parametric programming remains intractable, with challenges in real-time computation and storing closed-form solutions.
Numerous approximate methods have been proposed to address these issues. A hierarchical gridding scheme to construct low-complexity approximate control laws through selective storage of grid points with significant weights is introduced in \cite{summers2011multiresolution}. A parametric B$\And$B method to identify suboptimal solutions with guaranteed bounds is given in \cite{axehill2014parametric}.

Recent studies have increasingly turned to learning methods to quickly obtain sub-optimal solutions by making computationally expensive decisions offline. In \cite{masti2019learning}, a compact neural network for predicting binary solutions in multi-parametric MIQPs is introduced, enabling binary warm-starts in HCPs. In \cite{zhang2019safe}, MPC policies for hybrid systems are approximated using supervised learning, ensuring feasibility and near-optimality with high probability. An extended framework is introduced in \cite{zhang2020near}, with a backup controller that filters out approximated MPC policies that are not near-optimal. A non-parametric learning algorithm that captures the mode sequence of hybrid MPC solutions is proposed in \cite{zhu2020fast}, which facilitates warm-starts. The CoCo (Combinatorial Offline, Convex Online) framework is introduced in \cite{cauligi2020learning}, where a multi-class classifier is trained offline to learn strategies for the combinatorial part and online control decisions are made in milliseconds, solving a small convex optimization.  This approach is extended in \cite{cauligi2021coco} to handle discrete control variables. CoCo is further generalized to mixed-integer optimization (MIO) \cite{bertsimas2022online}. A strong solver generating near-optimal solutions with many iterations and a weak solver providing sub-optimal solutions quickly are proposed in \cite{chakrabarty2021learning}. A deep neural network is trained to learn the admissible subregions of the state-space, enabling the replacement of the strong solver with the weak solver without significant performance deterioration.  A similar approach is proposed in \cite{russo2023learning}, where a neural network predicts a strategy consisting of LPs that partition the feasible set and a candidate integer solution. Training these learning models can be computationally very demanding as the process requires solving a large number of HCPs with diverse parameter choices offline.


An alternative approach to improving solution times involves warm-start strategies to solve B$\And$B subproblems efficiently. Dual methods have been popular as the dual constraints remain constant, allowing simple objective modifications to solve QP subproblems in B$\And$B.
A dual QP algorithm utilizing gradient projection methods is proposed in \cite{axehill2008dual}, offering the advantages of warm-start and rapid identification of active sets.  An accelerated dual gradient projection method for QPs is introduced in \cite{naik2017embedded}, which is also efficient for embedded systems.
A tailored algorithm for small MIQPs in hybrid MPC, coupling a B$\And$B scheme with a robust QP solver \citep{bemporad2017numerically} based on nonnegative least squares (NNLS), was further developed in \cite{bemporad2018numerically}.
A B$\And$B algorithm using the OSQP solver \citep{stellato2020osqp} is presented in \cite{stellato2018embedded}, which supports warm-starts and factorization caching.
A warm-start B$\And$B algorithm that leverages the receding horizon of MPC, reusing search trees and dual bounds, is introduced in \cite{marcucci2020warm}. 
Another MIQP solver for embedded systems is developed \cite{arnstrom2023bnb}, utilizing a dual active-set solver \cite{arnstrom2022dual} in B$\And$B to enable warm-starts. An early termination strategy with efficient projections for interior point methods for QPs in B$\And$B is proposed in \cite{liang2020early}. 
A B$\And$B algorithm employing FORCES \citep{domahidi2012forces}, an interior point method tailored for multistage structures arising in embedded MPC, and heuristics to accelerate computations are given in \cite{frick2015embedded}.
The reliability branching, a mixture of strong branching and the pseudo-costs method, is proposed in \cite{hespanhol2019structure}. 
An exact block-sparse presolve techniques are given in \cite{quirynen2023tailored} to remove decision variables and inequality constraints in HCPs efficiently.

Another stream of research has been on strengthening the convex relaxations of MIQPs for hybrid control to improve the lower bounds and thereby reduce the B\&B search effort.  Three equivalent mixed-integer convex programming (MICP) formulations for HCPs are explored in \cite{axehill2010convex}, comparing QP, SDP, and equality-constrained SDP relaxations. The strength of big-M and extended convex hull formulations for piecewise affine (PWA) systems are evaluated in \cite{marcucci2019mixed}. 
In \cite{kurtz2021more}, the variable space is expanded by introducing additional binary variables for transitions at each period, resulting in a tight convex relaxation. In \cite{marcucci2024shortest}, the HCP is modeled as a generalized shortest path problem (SPP). \ignore{ defined as follows: Given a directed graph $G = (V,E)$, a random variable $w_v \in W_v$ indicates the position of vertex $v \in V$ in a nonempty compact convex set $W_v$.  Each edge length $\ell_e(w_u, w_v)$ is a convex function of the two end points $w_u$ and $w_v$ for $e=(u,v) \in E$.  HCP is then modeled as a shortest $s, t$ path problem among  the set of all $s-t$ paths $\Pi$:
\begin{equation*}
    \begin{aligned}
        \min_{\pi \in \Pi} \min_{w_{\pi} \in W_{\pi}} \ell_{\pi} (w_{\pi}).
    \end{aligned}
\end{equation*} 
\todo{SPP is unclear. Where did random variables come from? }
}
We compare our model with the SPP formulation in Section \ref{sec:experiment}.
In addition, strong formulations of HCPs for specific applications are also actively pursued, e.g., \citep{andrikopoulos2013piecewise}, \citep{han2017feedback}, \citep{deits2014footstep}, \citep{ aceituno2017simultaneous}. See \cite{ag-super}, \cite{han-2x2}, \cite{wei2024convex} and the references therein for recent convexification methods for convex quadratic programs with indicator variables.


These stronger convex formulations are developed in an extended space using auxiliary variables derived from disjunctive programming, which makes them challenging to deploy in large-scale applications. Limited research exists on generating strong convex relaxations of HCPs in the original variable space to enable tractable solution times. Addressing this gap in the literature is the primary goal of the current paper.

\ignore{ 
On the contrary, there is an active exploration of strengthening the relaxations of related MICPs.
\cite{frangioni2006perspective} introduce \textit{perspective cuts} based on the subdifferentials of the perspective function. \cite{akturk2009strong} develop a closed-form convex hull representation of a mixed-integer set, closely related to the feasible set of \eqref{ORIG_PROB}. \cite{gunluk2012perspective} and  \cite{bonami2015mathematical} demonstrated how the perspective formulation can be employed to derive strong formulations for some MICPs in the original space. Other studies \citep{anstreicher2021quadratic, wei2022convex, wu2017quadratic, dong2013valid, gunluk2010perspective, atamturk2018strong} focus on reformulations and cut generations for relevant MIQPs, but the key difference is that there is no linear system constraint that links between variables. 
Motivated by these studies, we utilize the problem structure of \eqref{ORIG_PROB} to generate effective cuts in the original variable space.
}

\section{Strong formulation of HCP} \label{sec:strong_form}

In this section, we derive a strong formulation for the HCP \eqref{ORIG_PROB}. 
Starting with the single-period HCP with one-dimensional state, control, and indicator variables in Section \ref{sec:cut_simple}, a conic quadratic convex hull representation of the epigraph set is constructed in an extended space. Then, two classes of cuts in the original variable space are derived. The cut generation is progressively generalized in the following subsections, ultimately making it applicable to the $n$-period HCP \eqref{ORIG_PROB}. 
Section \ref{sec:cut_general1} extends the approach to the multi-dimensional state and control variables case, $d_x, d_y >1$, and Section \ref{sec:cut_general2} demonstrates it for any $\bm{f}_t \in \mathbb{R}^{d_x}$ in the linear dynamics constraint.
In Section \ref{sec:cut_general3} we show how the cut generation can be adapted to the HCP with a multi-dimensional indicator variable $\bm{z}_t \in \{0,1\}^{d_z}$, $d_z >1$. Finally, in Section \ref{sec:cut_general4}, the cut generation is applied to the $n$-period HCP.

\subsection{Single-period HCP with one-dimensional variables}\label{sec:cut_simple}
Consider the epigraph formulation of the single-period HCP, where all variables are one-dimensional:
\begin{equation}
    \begin{aligned}
        \min \ \ &w + q_1 x_1^2\\
        \text{s.t. } \ &q_2 x_2^2 +ry^2 + z \leq w\\
        & x_2 = ax_1 + by + cz\\
        & g z \leq y \leq h z\\
        & \ell_{t} \leq x_t \leq u_{t}, \quad t=1,2\\
        & z \in \{0,1\}, \ x_1,x_2,y, w \in \mathbb{R}
    \end{aligned}
    \label{SIMPLE_PROB}
\end{equation}
with $q_1, q_2,r \geq 0$. Let  $\myZ_1$ denote the constraint set of \eqref{SIMPLE_PROB} and (R\ref{SIMPLE_PROB}) its continuous relaxation obtained by replacing the binary domain $\{0,1\}$ of $z$ with the interval $[0,1]$.

\begin{proposition} \label{prop:1dim_nonlinear_ext}
The convex hull of $\myZ_1$ can be represented in an extended space as:
\begin{subequations}
\begin{align}
    &\conv (\myZ_1) = \bigg\{ (x_1, x_2, y, z, w): \text{constraints of (R\ref{SIMPLE_PROB}) and }\qquad \nonumber\\
    &  \quad \exists p \in \bbR \text{: } w \geq q_2 \left(\frac{1}{z}-1\right)\left(x_2 - ap\right)^2 + q_2 x_2^2 +z + \frac{ry^2}{z}  \label{1dim_nonlinear_ext_cut} \\
    & \quad \ell_{1} \leq p \leq u_{1}, \ \frac{x_1- u_{1}z}{1-z} \leq p \leq \frac{x_1 - \ell_{1}z}{1-z}, \label{1dim_ext_bd1}\\ 
    & \quad \ell_{2} \leq a p \leq u_{2}, \ \frac{x_2 - u_{2} z}{1-z} \leq ap \leq \frac{x_2 - \ell_{2}z}{1-z}\bigg\} \cdot \label{1dim_ext_bd2}
\end{align}
\end{subequations}
\end{proposition}
\begin{proof}
The feasible set $\myZ_1$ is the union of $\myZ_1^0$ and $\myZ_1^1$ defined as
\begin{equation*}
    \begin{aligned}
     &\myZ_1^0 = \big\{\left(x_1^0, x_2^0, y^0 , 0, w^0\right): q_2\left(x_2^0\right)^2 \leq w^0,  \\
     &\pushright{x_2^0 = ax_1^0, \ y^0 = 0, \ \ell_{t} \leq x_t^0 \leq u_{t}, \ t=1,2\big\}, \ }\\
     &\myZ_1^1 = \big\{\left(x_1^1, x_2^1, y^1 , 1, w^1\right):  q_2 \left(x_2^1\right)^2 +r \left(y^1\right)^2 +1 \leq w^1, \qquad \\
     &\pushright{x_2^1 = ax_1^1 + by^1 + c, \ g\leq y^1 \leq h, \ \ell_{t} \leq x_t^1 \leq u_{t}, \ t=1,2\big\}.\ }
    \end{aligned}
\end{equation*}
Since $\myZ_1^0$ and $\myZ_1^1$ are convex, $\bm{v}=(x_1, x_2, y, z, w) \in \conv(\myZ_1)$ if and only if $\bm{v}$ is a convex combination of points in $\myZ_1^0$ and $\myZ_1^1$, i.e., there exist $\bm{v}^0 = \left(x_{1}^0,x_{2}^0,y^0,0,w^0\right) \in \myZ_1^0$, $\bm{v}^1 = \left(x_{1}^1,x_{2}^1,y^1,1,w^1\right) \in \myZ_1^1$ and $\lambda \in \left[0,1\right]$ such that
\begin{align*}
    &x_1 = (1-\lambda) x_1^0 + \lambda x_1^1, \ x_2 = (1-\lambda) x_2^0 + \lambda x_2^1, \\ &y = (1-\lambda) y^0 + \lambda y^1, \ z = \lambda, \ \delta = (1-\lambda) w^0 + \lambda w^1.
\end{align*}
Employing $\lambda = z$, $x_2^0 = ax_1^0$, $w^0 \geq q_2 \left( x_2^0 \right)^2$, and $w^1 \geq q_2 \left(x_2^1\right)^2 +r \left(y^1\right)^2 +1$, the point $v$ can be written as
\begin{align*}
    &x_1 = (1-z) x_1^0 + zx_1^1, \ x_2 = (1-z) ax_1^0 + zx_2^1, \ y= zy^1, \\ 
    &w \geq (1-z) q_2 a^2 \left( x_1^0 \right)^2 + z \left( q_2 \left( x_2^1 \right)^2 + r \left(y^1\right)^2 + 1 \right),
\end{align*}
where 
\begin{align*}
    x_2^1 = ax_1^1 + by^1 + c = a \left( \frac{x_1 - (1-z) x_1^0}{z}\right) + \frac{by}{z} + c.
\end{align*}
Projecting out $y^1$, $x_2^1$ and denoting $p = x_1^0$, constraint \eqref{1dim_nonlinear_ext_cut} is obtained along with the additional constraints
\begin{align*}
    &x_1^0 = p \in [\ell_{1}, u_{1}], \ x_1^1 = \frac{x_1 - z p}{1-z} \in [\ell_{1}, u_{1}], \\
    &x_2^0 = ap \in [\ell_{2}, u_{2}], \ x_2^1 = \frac{x_2 - (1-z) ap}{z} \in [\ell_{2}, u_{2}],
\end{align*}
which are equivalent to constraints in \eqref{1dim_ext_bd1} and \eqref{1dim_ext_bd2}.
\end{proof}

\begin{remark}
Constraint \eqref{1dim_nonlinear_ext_cut} of Proposition~\ref{prop:1dim_nonlinear_ext}
\[
w \geq q_2 \left(\frac{1}{z}-1\right)\left(x_2 - ap\right)^2 + q_2 x_2^2 +z + \frac{ry^2}{z}
\]
enhances the perspective reformulation of a convex quadratic function with an indicator \citep{akturk2009strong}, \citep{gunluk2012perspective} by exploiting the linear dynamics equality $x_2 = a x_1 + by + cz$. 
It is instructive to analyze \eqref{1dim_nonlinear_ext_cut} for values of $z \in [0,1]$. 
When $z=1$, \eqref{1dim_nonlinear_ext_cut} reduces to
$w \ge q_2 x_2^2 + r y^2 + 1$, 
which is the original quadratic constraint in \eqref{SIMPLE_PROB}. When $z=0$, $y = 0$ and, by \eqref{multidim_ext_bd2}, $ap = x_2$. Therefore, \eqref{1dim_nonlinear_ext_cut} reduces to $w \ge q_2 x_2^2$. For $0 < z < 1$,  \eqref{1dim_nonlinear_ext_cut} is stronger than the original constraint
$w \ge q_2 x_2^2 + r y^2 + z$.
\end{remark}

Defining $\Tilde{p} = (1-z) p$, the convex hull expression in Proposition~\ref{prop:1dim_nonlinear_ext} can be reformulated with convenient conic quadratic inequalities as shown next.

\begin{corollary} \label{cor:1dim_conv_ext}
A conic quadratic representation of $\conv(\myZ_1)$ can be formulated as follows:
\begin{align*}
    &\conv (\myZ_1) =  \Big\{ (x_1, x_2, y, z, w): \text{constraints in (R\ref{SIMPLE_PROB}) and } \\
    & \qquad \exists \tilde{p}, \tilde{w}_1, \tilde{w}_2 \in \bbR \text{ s.t. } w \geq  \tilde{w}_1 + \tilde{w}_2 + z, \\
    & \qquad \tilde{w}_1 z \geq q_2 (x_2 - a\tilde{p})^2+ ry^2, \ \Tilde{w}_2 (1-z) \geq q_2 a^2 \tilde{p}^2,  \\ 
    & \qquad \ell_{1} (1-z) \leq \tilde{p} \leq u_{1} (1-z), \  \ell_{1}z \leq x_1 -\Tilde{p} \leq u_{1} z, \\ 
    & \qquad \ell_{2} (1-z) \leq a\tilde{p} \leq u_{2} (1-z), \ \ell_{2}z \leq x_2 - a\Tilde{p} \leq u_{2} z\Big\} \cdot
\end{align*}
\end{corollary}
\noindent 
Corollary~\ref{cor:1dim_conv_ext} follows immediately from plugging in $p = \frac{\tilde{p}}{1-z}$ to the $\conv(\myZ_1)$ representation in Proposition~\ref{prop:1dim_nonlinear_ext}.

Although formulations in Proposition~\ref{prop:1dim_nonlinear_ext} and Corollary~\ref{cor:1dim_conv_ext} are convenient to state, they include additional variables, which we need to project out to arrive at a strong formulation of \eqref{SIMPLE_PROB} in the original variable space. To this end,
define $\ell_{a} =\min \left\{a\ell_{1}, au_{1}\right\}$, $u_{a} =\max \left\{a\ell_{1}, au_{1}\right\}$. For given $(x_1, x_2, y, z)$, consider the projection problem:
\begin{equation}
    \begin{aligned}
        \tau(x_1, x_2, y, z) = \ &\underset{p}{\min} \ \ (x_2 - ap)^2 \\
        &\text{ s.t.} \ \ \ell(x_1,x_2, z) \leq ap \leq u(x_1,x_2,z)
    \end{aligned}
    \label{MIN_PROB}
\end{equation}
where
\begin{align*}
    \ell(x_1,x_2, z) &:= \max \bigg\{\ell_{a},\ell_{2},\frac{x_2-u_{2}z}{1-z}, \frac{ax_1 - u_{a} z}{1-z}\bigg\},\\ 
    u(x_1,x_2, z) &:= \min \bigg\{u_{a},u_{2}, \frac{x_2 - \ell_{2}z}{1-z}, \frac{ax_1 -\ell_{a} z}{1-z}\bigg\} \cdot
\end{align*}

\noindent For \eqref{MIN_PROB} to be feasible, it is easy to see that the following eight conditions must hold:
\begin{enumerate}[label=(\alph*)]
    \item \label{itm:1} $\ell_{a} \leq u_{2}$
    \item \label{itm:2} $\ell_{2} \leq u_{a}$
    \item \label{itm:3} $\ell_{a}(1-z) \leq x_2 - \ell_{2} z$
    \item \label{itm:4} $x_2 - u_{2}z \leq u_{a} (1-z)$
    \item \label{itm:5} $x_2 \geq by + cz  + \ell_{a}z + \ell_{2}(1-z)$
    \item \label{itm:6} $x_2 \leq by + cz + u_{a}z + u_{2} (1-z)$
    \item \label{itm:7} $u_{2}z \geq by + cz + \ell_{a}z$
    \item \label{itm:8} $by + cz + u_{a}z \geq \ell_{2}z$
\end{enumerate}

\begin{proposition}\label{Prop1}
For any feasible solution of \eqref{SIMPLE_PROB}, 
\begin{enumerate}
    \item\label{prop1.1} if \ref{itm:1} or \ref{itm:2} is violated, $z =1$ holds,
    \item\label{prop1.2} otherwise, \ref{itm:3}--\ref{itm:8} hold. 
\end{enumerate}
\end{proposition}
\noindent 
We call \ref{itm:3}--\ref{itm:8} as the {\textit{feasibility cuts}.
The proof of Proposition \ref{Prop1} can be found in Appendix \ref{AppendixA}.
While the feasibility cuts are satisfied by all solutions of \eqref{SIMPLE_PROB}, they can remove points in its continuous relaxation with $0 < z < 1$. By Proposition \ref{Prop1}, if \ref{itm:1} or \ref{itm:2} is violated, $z =0$ is not possible and one can fix $z=1$. Otherwise, adding the feasibility cuts to \eqref{SIMPLE_PROB} ensures the feasibility of \eqref{MIN_PROB}, leading to the cut in the original variable space, referred to as the \textbf{\textit{nonlinear cut}}, in Corollary \ref{Prop2_revised}.

\begin{corollary}\label{Prop2_revised}
Let $\ell (x_1,z) = \max\left\{\ell_a, \frac{a x_1 - u_{a} z}{1-z}\right\}$, $u(x_1,z) = \min \big\{ u_{a}, \frac{ax_1 - \ell_a z}{1-z}\big\}$. The following \textbf{nonlinear cut} is valid for \eqref{SIMPLE_PROB}:
\begin{align}
    w \geq \ & q_2 \left(\frac{1}{z} - 1\right) \left( \left( \ell (x_1,z) - x_2 \right)_+^2 + \left( x_2 - u(x_1,z)\right)_+^2 \right) \nonumber\\
    &\ + q_2 x_2^2 + \frac{ry^2}{z} + sz \label{nonlinear_cut} 
\end{align}
\end{corollary}
\noindent Proof of Corollary~\ref{Prop2_revised} is given in Appendix~\ref{Appendix:nonlinear_proof}. Note that if $x_2 \in [\ell (x_1,z), u(x_1,z)]$, the first term in \eqref{nonlinear_cut} vanishes and inequality reduces to the perspective cut \citep{akturk2009strong}, \citep{gunluk2012perspective};  otherwise, it improves the perspective cut by utilizing the linear dynamics constraint.
Although the right-hand side of \eqref{nonlinear_cut} is a convex piecewise quadratic function, the cut cannot be directly added to \eqref{SIMPLE_PROB} while maintaining the convexity of the problem due to the boundaries $\frac{ax_1 - u_{a} z}{1-z} = x_2$ and $\frac{ax_1 - \ell_a z}{1-z} = x_2$. 
Therefore, to implement \eqref{nonlinear_cut} we resort to its linear underestimators.

\begin{corollary}\label{cor_grad_cut}
Given any point $\bar{\bm{\chi}} =(\bar{x}_1, \bar{x}_2, \bar{y},\bar{z})$ of (R\ref{SIMPLE_PROB})
satisfying the feasibility cuts, a convex quadratic cut
\begin{align*}
    w \geq \mu(\bar{\bm{\chi}}) := q_2 \left(\frac{1}{z} - 1\right) (\bar{x}_2 - \bar{\sigma})^2 + q_2 x_2^2 + z + \frac{ry^2}{z} 
\end{align*}
is valid for \eqref{SIMPLE_PROB} under the conditions listed in Table \ref{table:feasible_cut_revised} with the corresponding value of $\bar{\sigma}$. Moreover, the \textbf{\textit{gradient cut}} 
\begin{align}
    w \geq \ &\mu(\bar{\bm{\chi}})  + \nabla \mu(\bar{\bm{\chi}}) \left( \bm{\chi} - \bar{\bm{\chi}}\right) \label{feas_cut_grad_ver} 
\end{align}
is valid as well.
\end{corollary}
\begin{table}[ht]
\caption{Quadratic cuts for the one-dimensional case.}
\label{table:feasible_cut_revised}
\captionsetup{justification=centering}
\centering
\input{tables/Feasible_cuts_arXiv}
\centering
\vspace{2mm}
($\star$) : Only when $\ell_{a} \geq \ell_{2}$ \quad 
($\dagger$) : Only when $u_{a} \leq u_{2}$
\end{table}

The cut-generation process for the one-period and one-dimensional case in this section forms the basis for the subsequent generalizations.
\noindent Note that all results in this section hold for $d_y \ge 1$ as well.

\subsection{Multi-dimensional state and control variables} \label{sec:cut_general1}

We now extend the cut-generation process to the single-period HCP with multi-dimensional state and control variables, $d_x, d_y \ge 1$. For nonnegative diagonal matrix $\bm{Q}_2$ and $\bm{Q}_1, \bm{R} \succeq 0$, consider 
\begin{equation}
    \begin{aligned}
    \min \ \ &w + \bm{x}_{1}^\top \bm{Q}_1 \bm{x}_{1}\\
    \text{s.t. } \ & \bm{x}_2^{\top}\bm{Q}_2 \bm{x}_2 + \bm{y}^{\top}\bm{R} \bm{y} + z  \leq w\\
    & \bm{x}_2= \bm{A}\bm{x}_{1} + \bm{B} \bm{u} + \bm{c} z,\\
    & \bm{g} z \leq \bm{y} \leq \bm{h} z,\\
    & \bm{\ell}_{t} \leq \bm{x}_{t} \leq \bm{u}_{t}, \quad t = 1,2 \\
    & z \in \{0,1\},\ \bm{x}_{1},\bm{x}_2 \in \mathbb{R}^{d_x}, \bm{y} \in \mathbb{R}^{d_y}, \ w \in \mathbb{R}.
    \end{aligned}
    \label{Form:general case}
\end{equation}
Denote $\myY_1$ as the feasible set and (R\ref{Form:general case}) as the continuous relaxation of \eqref{Form:general case} with $z \in [0,1]$.

\pagebreak

The convex hull representation in Proposition~\ref{prop:1dim_nonlinear_ext} and its conic quadratic reformulation in Corollary~\ref{cor:1dim_conv_ext} can be extended to the multi-dimensional case as follows.
\begin{corollary} \label{cor:multidim_nonlinear_ext}
The convex hull $\myY_1$ can be stated as:
\begin{subequations}
\begin{align}
    &\conv (\myY_1) = \Big\{ (\bm{x}_{1}, \bm{x}_2, \bm{y}, z, w): \text{constraints in (R\ref{Form:general case}) and }   \nonumber\\
    & \quad \exists \bm{p} \in \bbR^{d_x} \text{ s.t. } w \geq \left(\frac{1}{z}-1\right) \left(\bm{x}_2 - \bm{A} \bm{p}\right)^\top \bm{Q}_2 \left(\bm{x}_2 - \bm{A} \bm{p}\right) \nonumber \\
    &\qquad \qquad \qquad \qquad \quad + \bm{x}_2^\top \bm{Q}_2 \bm{x}_2 +z + \frac{1}{z} \bm{y}^\top \bm{R} \bm{y}  \label{multidim_nonlinear_ext_cut} \\
    & \qquad \bm{\ell}_{1} \leq \bm{p} \leq \bm{u}_{1}, \ \frac{\bm{x}_{1}- \bm{u}_{1}z}{1-z} \leq \bm{p} \leq \frac{\bm{x}_{1} - \bm{\ell}_{1}z}{1-z}, \label{multidim_ext_bd1}\\ 
    & \qquad \bm{\ell}_{2} \leq \bm{A} \bm{p} \leq \bm{u}_{2}, \ \frac{\bm{x}_2 - \bm{u}_{2} z}{1-z} \leq \bm{A} \bm{p} \leq \frac{\bm{x}_2 - \bm{\ell}_{2}z}{1-z}\Big\}. \label{multidim_ext_bd2}
\end{align}
\end{subequations}
Moreover, $\conv (\myY_1)$  can be reformulated using conic quadratic inequalities after defining $\tilde{\bm{p}} = (1-z) \bm{p} \in \bbR^{d_x}$:
\begin{align*}
    &\conv (\myY_1) =  \Big\{ (\bm{x}_{1}, \bm{x}_2, \bm{y}, z, w): \text{constraints in (R\ref{Form:general case}) and } \\
    & \qquad \exists \tilde{\bm{p}} \in \bbR^{d_x}, \ \tilde{w}_1, \tilde{w}_2 \in \bbR \text{ s.t. } w \geq  \tilde{w}_1 + \tilde{w}_2 + z, \\
    & \qquad \tilde{w}_1 z \geq (\bm{x}_2 - \bm{A}\tilde{\bm{p}})^\top \bm{Q}_2 (\bm{x}_2 - \bm{A}\tilde{\bm{p}}) + \bm{y}^\top \bm{R} \bm{y}, \\
    & \qquad \Tilde{w}_2 (1-z) \geq \tilde{\bm{p}}^\top \bm{A}^\top \bm{Q}_2 \bm{A} \tilde{\bm{p}}, \\ 
    & \qquad \bm{\ell}_{1} (1-z) \leq \tilde{\bm{p}} \leq \bm{u}_{1} (1-z), \  \bm{\ell}_{1}z \leq \bm{x}_{1} -\Tilde{\bm{p}} \leq \bm{u}_{1} z, \\ 
    & \qquad \bm{\ell}_{2} (1-z) \leq \bm{A}\tilde{\bm{p}} \leq \bm{u}_{2} (1-z), \ \bm{\ell}_{2}z \leq \bm{x}_2 - \bm{A}\Tilde{\bm{p}} \leq \bm{u}_{2} z\Big\}.
\end{align*}
\end{corollary}


As in the previous section, the auxiliary variable $\bm{p}$ can be projected out as follows:
\begin{align*}
    &w \geq \left(\frac{1}{z}-1\right)\tau(\bm{x}_{1},\bm{x}_2, \bm{y},z) + \bm{x}_2^{\top}\bm{Q}_2\bm{x}_2 + z + \frac{\bm{y}^{\top}\bm{R} \bm{y}}{z}
\end{align*}
where $\tau(\bm{x}_{1},\bm{x}_2, \bm{y},z)$ is the optimal value of the problem
\begin{equation}
    \begin{aligned}
      \underset{\bm{p} \in \mathbb{R}^{d_x}}{\min} \ & (\bm{x}_2 - \bm{A} \bm{p})^{\top} \bm{Q}_2 (\bm{x}_2 - \bm{A} \bm{p} )\\
    \text{s.t. } \ & \eqref{multidim_ext_bd1}, \eqref{multidim_ext_bd2}.
    \end{aligned}
    \label{Form:general sub prob}
\end{equation}
However, a closed-form optimal solution $\bm{p}^*$ cannot be obtained as the bounds on $\bm{A} \bm{p}$ may not be decoupled for $p$. Therefore, we resort to the following relaxation. Define $\upperh, \ \lowerh: \mathbb{R}^{d_x}\times \mathbb{R}^{d_x} \rightarrow \mathbb{R}^{d_x}$ as
\begin{align*}
    \left(\upperh(\bm{u}, \bm{\ell} )\right)_i := \sum\limits_{j:A_{ij} > 0} A_{ij}u_{j}+ \sum\limits_{k:A_{ik}<0} A_{ik} \ell_{k},\\
    \left(\lowerh(\bm{u}, \bm{\ell})\right)_i := \sum\limits_{j:A_{ij}>0} A_{ij}\ell_{j}+ \sum\limits_{k:A_{ik}<0} A_{ik}u_{k},
\end{align*}
for $i \in \left[d_x\right]$, and let $\bm{u}_{\bm{A}}: = \upperh(\bm{u},\bm{\ell})$ and $\bm{\ell}_{\bm{A}} = \lowerh(\bm{u},\bm{\ell})$. Consider two sets $S_1$ and $S_2$ defined as
\begin{align*}
    S_1 := \{ \bm{p} : \bm{\ell}_{1} \leq \bm{p} \leq \bm{u}_{1} \}, \ S_2 := \{\bm{p}: \bm{\ell}_{\bm{A}} \leq \bm{A} \bm{p} \leq \bm{u}_{\bm{A}}\}.
\end{align*}
While $S_1 \subseteq S_2$ holds, $S_2 \subseteq S_1$ may not be true. Therefore, replacing \eqref{multidim_ext_bd1} with 
\begin{equation*}
    \begin{aligned}
        \bm{\ell}_{\bm{A}} \leq \bm{A} \bm{p} \leq \bm{u}_{\bm{A}}, \ \frac{\bm{A}\bm{x}_{1} - \bm{u}_{\bm{A}}z}{1-z} \leq \bm{A}  \bm{p} \leq \frac{\bm{x}_{1} - \bm{\ell}_{\bm{A}}z}{1-z},
    \end{aligned}
\end{equation*}
gives a relaxation of \eqref{Form:general sub prob}. As such, 
define 
\begin{align*}
    \bm{\ell} (\bm{x}_{1},\bm{x}_2,z) &:= \max \bigg\{\bm{\ell}_{\bm{A}},\bm{\ell}_{2},\frac{\bm{A}\bm{x}_{1} - \bm{u}_{\bm{A}}z}{1-z}, \frac{\bm{x}_2 - \bm{u}_{2}z}{1-z} \bigg\}\\
    \bm{u}(\bm{x}_{1},\bm{x}_2,z) &:= \min \bigg\{\bm{u}_{A},\bm{u}_{2},\frac{\bm{A}\bm{x}_{1} - \bm{\ell}_{\bm{A}}z}{1-z}, \frac{\bm{x}_2 - \bm{\ell}_{2}z}{1-z} \bigg\}
\end{align*}
where the max/min functions are applied elementwise, and let $\bm{\sigma} = \bm{A} \bm{p}$. Consider the relaxation of \eqref{Form:general sub prob}
\begin{equation}
    \begin{aligned}
        \underset{\bm{\sigma} \in \mathbb{R}^{d_x}}{\min} \ &(\bm{x}_2 - \bm{\sigma})^{\top} \bm{Q}_2 (\bm{x}_2 - \bm{\sigma})^{\top}\\
        \text{s.t.} \quad &\bm{\ell} (\bm{x}_{1},\bm{x}_2,z) \leq \bm{\sigma} \leq \bm{u}(\bm{x}_{1},\bm{x}_2,z), \\
        & \bm{\sigma} = \bm{A} \bm{x} \hspace{1mm} \Leftrightarrow \hspace{1mm} \bm{\sigma} \in C(\bm{A})
    \end{aligned}
    \label{Form:general relaxed sub prob}
\end{equation}
where $C(\bm{A})$ represents the column space of $\bm{A}$. If $\bm{A}$ is full rank, $\bm{\sigma} \in C(\bm{A})$ can be ignored. Even if $\bm{A}$ is not full rank, removing $\bm{\sigma} \in C(\bm{A})$ still gives a relaxation of \eqref{Form:general relaxed sub prob}. Once $\bm{\sigma} \in C(\bm{A})$ is relaxed, \eqref{Form:general relaxed sub prob} can be decomposed since $\bm{Q}_2$ is diagonal. Thus, we obtain 
\begin{align}
    \tau'(\bm{x}_{1},\bm{x}_2,z) = \sum_{i=1}^{d_x} \  &\underset{\sigma_i \in \bbR}{\min}  \ q_i(x_{2,i} - \sigma_i)^2 \label{Form:sep general relaxed sub prob}\\
    &\text{ s.t.} \ \ \ell (\bm{x}_{1},\bm{x}_2,z)_i \leq \sigma_i \leq u (\bm{x}_{1},\bm{x}_2,z)_i,  \nonumber
\end{align}
where $q_i:=\diag(\bm{Q}_2)_i$ for $i \in [d_x]$.
Each minimization in \eqref{Form:sep general relaxed sub prob} is in the form of \eqref{MIN_PROB}. Therefore, the same cut-generation process can be employed. 


\subsection{Linear system constraint with a constant vector} \label{sec:cut_general2}
Until now, we have considered the linear dynamics
\begin{align*}
    \bm{x}_2 = \bm{A}\bm{x}_{1} + \bm{B} \bm{y} + \bm{c} z+\bm{f},
\end{align*}
with $\bm{f} = \bm{0}$. We demonstrate here that the same approach can be applied for any $\bm{f} \in \mathbb{R}^{d_x}$. For simplicity of notation, consider the case of $d_x, d_y =1$. Again, with the disjunctive programming, the feasible set $\myZ_1$ can be represented as the union of two convex sets as in Section \ref{sec:cut_simple}. The two sets are equivalent to $\myZ_1^0$ and $\myZ_1^1$, except that the linear system constraints are revised to $x_{2}^0=ax_{1}^0 + f$  and $x_{2}^1 = ax_{1}^1 + by^1 + c + f$, respectively. 
Applying projections, we get
\begin{align*}
    w \geq q_2\bigg(\frac{1}{z}-1\bigg)\big(x_2 - ax_{1}^0 - d\big)^2 + q_2 x_2^2 +z + \frac{ry^2}{z},
\end{align*} 
with bounds $\ell (\bm{\chi}) \leq ax_{1}^0 + f \leq u(\bm{\chi})$ for $\bm{\chi}=(x_1, x_2, z)$,
\begin{align*}
    \ell (\bm{\chi}) {:=} \max \bigg\{\ell_{a}+f,\ell_{2},\frac{x_2-u_{2}z}{1-z}, \frac{(ax_1 +f) - \left(u_{a}+f\right) z}{1-z}\bigg\} \ \\
    u(\bm{\chi}) {:=} \min \bigg\{u_{a}+f,u_{2}, \frac{x_2 - \ell_{2}z}{1-z}, \frac{(ax_1 +f) -\left(\ell_{a}+f\right) z}{1-z}\bigg\} \cdot
\end{align*}
After replacing $ax_1$, $\ell_{a}$, $u_{a}$ with $ax_1+f$, $\ell_{a}+f$, $u_{a}+f$, respectively, the cut generation remains the same as in Section \ref{sec:cut_simple}. The same reasoning applies to the case of multi-dimensional state and control variables as well.

\subsection{Multi-dimensional binary decision variable
} \label{sec:cut_general3}
In this section, we further generalize the cut generation procedure to the HCP with multi-dimensional indicators $z \in \{0,1\}^{d_z}, \ d_z \ge 1$. Consider
\begin{equation}
    \begin{aligned}
    \min \ &w + \bm{x}_{1}^\top \bm{Q}_1 \bm{x}_{1}\\
    \text{s.t. } & \bm{x}_2^{\top}\bm{Q}_2 \bm{x}_2 + \bm{y}^{\top} \bm{R} \bm{y} + \bm{z}^\top \bm{S} \bm{z}  \leq w\\
    & \bm{x}_2= \bm{A}\bm{x}_{1} + \bm{B} \bm{y}+\bm{C} \bm{z},\\
    & \mathbbm{1}^\top \bm{z} \leq 1,\\
    & \bm{G} \bm{z} \leq \bm{y} \leq \bm{H} \bm{z},\\
    & \bm{\ell}_{t} \leq \bm{x}_{t} \leq \bm{u}_{t}, \quad t=1,2\\
    & \bm{z} \in \{0,1\}^{d_z}, \ \bm{x}_{1}, \bm{x}_2 \in \bbR^{d_x}, \ \bm{y} \in \bbR^{d_y}, w \in \bbR. \hspace{-6mm}
    \end{aligned}
    \label{Form:multi-dz}
\end{equation}
Let $\myX_1$ be the feasible set of \eqref{Form:multi-dz}.

Two disjunction methods are considered. A direct approach is to partition $\myX_1$ into ($d_z +1$)-disjunctions for each possible value of $\bm{z}$ and the other method is to split $K:=\{1,\ldots,d_z\}$ into two subsets. 

\vskip 1mm

\subsubsection{(\texorpdfstring{$d_z + 1$}{dz+1})-way disjunction} \label{sec:cut_general3_dz+1} 
The feasible set $\myX_1$ is the union of ($d_z + 1$) convex sets that each corresponds to the feasible set of \eqref{Form:multi-dz} for a particular $\bm{z}$ value: $X_1^0$ when $\bm{z}=\bm{0}$ and $Z_1^k$ when $\bm{z}=\bm{e}_k$, $k \in K$. Then, $(\bm{x}_{1}, \bm{x}_2, \bm{y}, \bm{z},w) \in \conv(\myX_1)$ can be expressed as a convex combination of $(\bm{x}_{1}^k, \bm{x}_2^k, \bm{y}^k , \bm{z}^k, w^k) \in X_1^k$, $k \in \{0\} \cup K$: $\exists \lambda^0, \lambda^1, \ldots, \lambda^{d_z} \in [0,1]$ such that $\sum_{k=0}^{d_z} \lambda^k = 1$ and
\begin{align*}
    (\bm{x}_{1}, \bm{x}_2, \bm{y}, \bm{z},w) = \sum_{k=0}^{d_z} \lambda^k (\bm{x}_{1}^k, \bm{x}_2^k, \bm{y}^k , \bm{z}^k, w^k).
\end{align*}
Then a valid cut in an extended space can be generated 
\begin{equation*}
    \begin{aligned}
        &w \geq \sum_{k \in K} \bm{z}_k \Big((\bm{A}\bm{x}_{1}^k + \bm{B} \bm{y}^k + \bm{c}_k)^\top \bm{Q}_2 (\bm{A}\bm{x}_{1}^k + \bm{B} \bm{y}^k +\bm{c}_k ) \\
        &\hspace{1.6cm} + {\bm{y}^k}^\top \bm{R} \bm{y}^k + S_{kk} \Big) + (1- \mathbbm{1}^\top \bm{z}) {(\bm{A} \bm{x}_{1}^0)}^\top \bm{Q}_2 (\bm{A} \bm{x}_{1}^0).
    \end{aligned}
\end{equation*}
We have no simple way of projecting this cut to the original variable space. Instead, the extended formulation needs to be directly solved. The ($d_z+1$)-way disjunction method yields a tighter relaxation bound but requires many auxiliary binary variables, possibly leading to long computation. The SPP model introduced in \cite{marcucci2024shortest}, which is further discussed in Section \ref{sec:experiment_disc}, employs this approach for every two consecutive periods.

\vskip 1mm

\subsubsection{Two-way disjunction} \label{sec:cut_general3_2way} 
Split the index set $K$ into two sets, $K_1 \subset K$ and $K_2 = K\backslash K_1$. 
Let $J = \{1,\ldots, d_y\}$ and define $J_1 = J \backslash \left\{ j \in J: y_j = 0 \text{ if } z_k = 0, \ \forall k \in K_2\right\}$, $J_2 = J \backslash \left\{ j \in J: y_j = 0 \text{ if } z_k = 0, \ \forall k \in K_1\right\}$. Note that $(J_1, J_2)$ need not be a partitioning of $J$. Denote $\bm{z}\kf \in \{0,1\}^{|K_1|}$ and $\bm{z}\ks \in \{0,1\}^{|K_2|}$ as the partial vectors of $\bm{z}$ with indices in $K_1$ and $K_2$, respectively. Define $\bm{y}\jf \in \mathbb{R}^{|J_1|}$ and $\bm{y}\js \in \mathbb{R}^{|J_2|}$ in the same manner. Then, the bound constraints on $\bm{y}$ can be split into bounds on  $\bm{y}\jf \in \mathbb{R}^{|J_1|}$ and $\bm{y}\js \in \mathbb{R}^{|J_2|}$
\begin{align*}
    \bm{G} \bm{z} \leq \bm{y} \leq \bm{H} \bm{z} \ \Leftrightarrow \ \makecell{\bm{G}\jf \bm{z}\kf \leq \bm{y}\jf \leq \bm{H}\jf \bm{z}\kf \\ \bm{G}\js \bm{z}\ks \leq \bm{y}\js \leq \bm{H}\js \bm{z}\ks}
    \end{align*}
where $\bm{G}_{J_i}, \bm{H}_{J_i} \in \mathbb{R}^{|J_i| \times |K_i|}$ are submatrices of $\bm{G}, \bm{H}$ with $(j,k)$-th elements $\forall j \in J_i, k \in K_i$ for $i=1,2$. 
Then, let $\bm{B}_{J_i} \in \mathbb{R}^{d_x \times |J_i|}$ be a submatrix of $\bm{B}$ consisting of the $j$-th columns for $j \in J_i$, and let $\bm{R}_{J_i} $ be the principal submatrix induced by set $J_i$, for $i=1,2$. Likewise, $\bm{C}_{K_i} \in \mathbb{R}^{d_x \times |K_i|}$ is a submatrix of $\bm{C}$ consists of $k$-th columns for $\forall k \in K_i$, and $\bm{S}_{K_i}$ is the principal submatrix of $\bm{S}$ defined by indices in $K_i$, for $i=1,2$. Then, \eqref{Form:multi-dz} can be reformulated as 
\begin{equation}
    \begin{aligned}
        \min \ &w + \bm{x}_{1}^\top \bm{Q}_1 \bm{x}_{1}\\
        \text{s.t.} \ &w \geq \bm{x}_2^\top \bm{Q}_2 \bm{x}_2 + \bm{y}\jf^\top \bm{R}\jf \bm{y}\jf + \bm{y}\js^\top \bm{R}\js \bm{y}\js \\
        & \qquad + \bm{z}\kf^\top \bm{S}\kf \bm{z}\kf + \bm{z}\ks^\top \bm{S}\ks \bm{z}\ks \\
        & \bm{x}_2 = \bm{A}\bm{x}_{1} +\bm{B}\jf \bm{y}\jf + \bm{B}\js \bm{y}\js + \bm{C}\kf \bm{z}\kf + \bm{C}\ks \bm{z}\ks \\
        & \bm{G}_{J_i} \bm{z}_{K_i} \leq \bm{y}_{J_i} \leq \bm{H}_{J_i} \bm{z}_{K_i}, \quad i=1,2\\
        & \bm{\ell}_{t} \leq \bm{x}_t \leq \bm{u}_{t}, \quad t=1,2\\
        & \mathbbm{1}_{|K_1|}^\top \bm{z}\kf + \mathbbm{1}_{|K_2|}^\top \bm{z}\ks \leq 1\\
        & \bm{z}\kf \in \{0,1\}^{|K_1|}, \ \bm{z}\ks \in \{0,1\}^{|K_2|} \\
        & \bm{y}\jf \in \mathbb{R}^{|J_1|}, \ \bm{y}\js \in \mathbb{R}^{|J_2|}\\
        & \bm{x}_{1},\bm{x}_2 \in \mathbb{R}^{d_x}, \ w \in \mathbb{R}.
    \end{aligned}
    \label{multi-dz 2-way}
\end{equation}
Let $\bar{\myX}_1$ denote the feasible set of \eqref{multi-dz 2-way} and
(R\ref{multi-dz 2-way}) be its convex relaxation with $\bm{z} \in [0,1]^{d_z}$.
Additionally, let (R\ref{multi-dz 2-way}-2) be another relaxation of \eqref{multi-dz 2-way} with $\bm{z}\kf \in [0,1]^{|K_1|}$ and $\bm{z}\ks \in [0,1]^{|K_2|}$, while $\lambda := \mathbbm{1}_{|K_1|}^\top \bm{z}\kf \in \{0,1\}$. Denote its feasible set as $\widetilde{\myX}_1$. 
Then, $\myX_1 \subseteq \widetilde{\myX}_1 \subseteq \bar{\myX}_1$ for any $K_1 \subseteq K$. We will generate cuts that is tight for $\Tilde{\myX}_1$, so that they cut off points in $\bar{\myX}_1 \backslash \Tilde{\myX}_1$.

A convex hull representation of $\Tilde{\myX}_1$ can be formulated in a similar manner as in Corollary~\ref{cor:multidim_nonlinear_ext}.

\begin{corollary} \label{cor:conv_tilde_Z1}
The convex hull of $\tilde{\myX}_1$ is described as
\begin{subequations}
\begin{align}
    &\conv\left(\tilde{\myX}_1\right) = \Bigg\{ (\bm{x}_{1}, \bm{x}_2, \bm{y}\jf, \bm{y}\js, \bm{z}\kf, \bm{z}\ks, w): \nonumber\\
    &\ \text{constraints in (R\ref{multi-dz 2-way}) and } \exists \bm{p} \in \bbR^{d_x} \text{ s.t. } \nonumber\\
    & \   \sigma = \bm{A} \bm{p} + \frac{\bm{B}\js \bm{y}\js+ \bm{C}\ks \bm{z}\ks}{1-\lambda}, \ \lambda = \mathbbm{1}_{|K_1|}^\top \bm{z}\kf, \label{sigma_def_multi_dz}\\
    &\ w \geq \left(\frac{1}{\lambda}-1\right) \left(\bm{x}_2 - \bm{\sigma} \right)^\top \bm{Q}_2 \left(\bm{x}_2 - \bm{\sigma} \right) + \bm{x}_2^\top \bm{Q}_2 \bm{x}_2 \label{multi_dz_nonlinear_ext_cut} \\
    &\qquad + \frac{\bm{y}\jf^\top \bm{R}\jf \bm{y}\jf + \bm{z}\kf^\top \bm{S}\kf \bm{z}\kf}{\lambda} + \frac{\bm{y}\js^\top \bm{R}\js \bm{y}\js + \bm{z}\ks^\top \bm{S}\ks \bm{z}\ks}{1-\lambda} , \nonumber \\
    & \ \bm{\ell}_{1} \leq \bm{p} \leq \bm{u}_{1}, \ \frac{\bm{x}_{1}- \bm{u}_{1}\lambda}{1-\lambda} \leq \bm{p} \leq \frac{\bm{x}_{1} - \bm{\ell}_{1} \lambda}{1-\lambda}, \label{multi_dz_ext_bd1}\\ 
    & \ \bm{\ell}_{2} \leq \bm{\sigma} \leq \bm{u}_{2}, \ \frac{\bm{x}_2 - \bm{u}_{2} \lambda}{1-\lambda} \leq \bm{\sigma} \leq \frac{\bm{x}_2 - \bm{\ell}_{2}\lambda}{1-\lambda}\Bigg\}, \label{multi_dz_ext_bd2}
\end{align}
\end{subequations}
for $K_1 \subseteq K$, and can be reformulated using conic quadratic inequalities by a change of the variable $\Tilde{\bm{p}} = (1-\lambda) \bm{p}$ as follows:
\begin{align*}
    &\conv (\Tilde{\myX}_1) = 
     \Big\{ (\bm{x}_{1}, \bm{x}_2, \bm{y}\jf, \bm{y}\js, \bm{z}\kf, \bm{z}\ks, w): \text{constrs. in (R\ref{multi-dz 2-way})} \\
    & \text{   and }\exists \tilde{\bm{p}} \in \bbR^{d_x}, \ \tilde{w}_1, \tilde{w}_2 \in \bbR \text{ s.t. } w \geq  \tilde{w}_1 + \tilde{w}_2, \\
    & \quad \tilde{\bm{\sigma}} = \bm{A}\tilde{\bm{p}} + \bm{B}\js \bm{y}\js + \bm{C}\ks \bm{z}\ks, \ \lambda = \mathbbm{1}_{|K_1|}^\top \bm{z}\kf, \\
    & \quad \tilde{w}_1 \lambda \geq (\bm{x}_2 - \tilde{\bm{\sigma}})^\top \bm{Q}_2 (\bm{x}_2 - \tilde{\bm{\sigma}})+ \bm{y}\jf^\top \bm{R}\jf \bm{y}\jf + \bm{z}\kf^\top \bm{S}\kf \bm{z}\kf, \\
    & \quad \Tilde{w}_2 (1-\lambda) \geq \tilde{\bm{\sigma}}^\top \bm{Q}_2 \tilde{\bm{\sigma}} + \bm{y}\js^\top \bm{R}\js \bm{y}\js + \bm{z}\ks^\top \bm{S}\ks \bm{z}\ks, \\ 
    & \quad \bm{\ell}_{1} (1-\lambda) \leq \tilde{\bm{p}} \leq \bm{u}_{1} (1-\lambda), \  \bm{\ell}_{1}\lambda \leq \bm{x}_{1} -\Tilde{\bm{p}} \leq \bm{u}_{1} \lambda, \\ 
    & \quad \bm{\ell}_{2} (1-\lambda) \leq \tilde{\bm{\sigma}} \leq \bm{u}_{2} (1-\lambda), \ \bm{\ell}_{2}\lambda \leq \bm{x}_2 - \Tilde{\bm{\sigma}} \leq \bm{u}_{2} \lambda\Big\}.
\end{align*}
\end{corollary}
\begin{proof}
The set $\Tilde{\myX}_1$ is the union of the two disjunctions $\Tilde{\myX}_1^0$ and $\Tilde{\myX}_1^1$ defined as
\begin{equation*}
    \begin{aligned}
     &\tilde{\myX}_1^0 = \big\{\left(\bm{x}_{1}^0, \bm{x}_2^0, \bm{y}\jf^0, \bm{y}\js^0 , \bm{z}\kf^0, \bm{z}\ks^0,  w^0\right): \bm{y}\jf^0 = 0, \ \bm{z}\kf^0 = 0, \\
     &\pushright{w^0 \geq \left(\bm{x}_2^0\right)^\top \bm{Q}_2 \bm{x}_2^0 + \left(\bm{y}\js^0\right)^\top \bm{R}\js \bm{y}\js^0 + \left(\bm{z}\ks^0\right)^\top \bm{S}\ks \bm{z}\ks^0, \ } \\
    &\pushright{ \bm{x}_2^0 = \bm{A}\bm{x}_{1}^0 +\bm{B}\js \bm{y}\js^0 + \bm{C}\ks \bm{z}\ks^0, \ \bm{G}\js \bm{z}\ks^0 \leq \bm{y}\js^0 \leq \bm{H}\js \bm{z}\ks^0, \ }\\
    &\pushright{\bm{\ell}_{1} \leq \bm{x}_{1}^0 \leq \bm{u}_{1}, \ \bm{\ell}_{2} \leq \bm{x}_2^0 \leq \bm{u}_{2}, \ \mathbbm{1}_{|K_2|}^\top \bm{z}\ks^0 \leq 1 \big\},}\\
     &\tilde{\myX}_1^1 = \big\{\left(\bm{x}_{1}^1, \bm{x}_2^1, \bm{y}\jf^1, \bm{y}\js^1 , \bm{z}\kf^1, \bm{z}\ks^1,  w^1\right): \bm{y}\js^1 = 0, \ \bm{z}\ks^1 = 0, \\
     &\pushright{w^1 \geq \left(\bm{x}_2^1\right)^\top \bm{Q}_2 \bm{x}_2^1 + \left(\bm{y}\jf^1\right)^\top \bm{R}\jf \bm{y}\jf^1 + \left(\bm{z}\kf^1\right)^\top \bm{S}\kf \bm{z}\kf^1, \ } \\
    &\pushright{ \bm{x}_2^1 = \bm{A}\bm{x}_{1}^1 +\bm{B}\jf \bm{y}\jf^1 + \bm{C}\kf \bm{z}\kf^1, \ \bm{G}\jf \bm{z}\kf^1 \leq \bm{y}\jf^1 \leq \bm{H}\jf \bm{z}\kf^1, \ }\\
    &\pushright{\bm{\ell}_{1} \leq \bm{x}_{1}^1 \leq \bm{u}_{1}, \ \bm{\ell}_{2} \leq \bm{x}_2^1 \leq \bm{u}_{2}, \ \mathbbm{1}_{|K_1|}^\top \bm{z}\kf^1 = 1 \big\}.}
    \end{aligned}
\end{equation*}
Then, for $\bm{\chi}=(\bm{x}_{1}, \bm{x}_2, \bm{y}\jf, \bm{y}\js, \bm{z}\kf, \bm{z}\ks) \in \conv\left(\tilde{\myX}_1\right)$ such that $\lambda = \mathbbm{1}_{|K_1|}^\top \bm{z}\kf \in (0,1)$, it holds $\bm{\chi} = (1-\lambda) \bm{\chi}^0 + \lambda \bm{\chi}^1$ for $\bm{\chi}^0 \in \tilde{\myX}_1^0$ and $\bm{\chi}^1 \in \tilde{\myX}_1^1$. Projecting out all auxiliary variables other than $\bm{p} = \bm{x}_{1}^0$, a convex hull representation of $\conv\left(\tilde{\myX}_1\right)$ in an extended space is obtained. 
\end{proof}
\noindent Note that $\lambda, \bm{\sigma}, \tilde{\bm{\sigma}}$ are used only for simplicity. 


To further project out the auxiliary variable $\bm{p}$ in the convex hull representation of $\Tilde{\myX}_1$, \eqref{multi_dz_nonlinear_ext_cut} is replaced with 
\begin{align*}
    w \geq &\bigg(\frac{1}{\lambda} - 1 \bigg) \tau(\bm{\chi}) + \bm{x}_2^\top \bm{Q}_2 \bm{x}_2  \\ 
    &\ + \frac{\bm{y}\jf^\top \bm{R}\jf \bm{y}\jf + \bm{z}\kf^\top \bm{S}\kf \bm{z}\kf}{\lambda} + \frac{\bm{y}\js^\top \bm{R}\js \bm{y}\js + \bm{z}\ks^\top \bm{S}\ks \bm{z}\ks}{1-\lambda} 
\end{align*}
where $\bm{\chi}=(\bm{x}_{1}, \bm{x}_2, \bm{y}\jf, \bm{y}\js, \bm{z}\kf, \bm{z}\ks)$ and
\begin{equation}
    \begin{aligned}
        \tau(\bm{\chi}) = \min_{\bm{p},\bm{\sigma}} \ & (\bm{x}_2 - \bm{\sigma} )^\top \bm{Q}_2 (\bm{x}_2 - \bm{\sigma} ) \\
        \text{ s.t. }  & \eqref{sigma_def_multi_dz}, \eqref{multi_dz_ext_bd1}, \eqref{multi_dz_ext_bd2}.
    \end{aligned} \label{multi_dz_subproblem}
\end{equation}
Then, replacing the bound constraints on \eqref{multi_dz_ext_bd1} with bound constraints on $\sigma$ using $\bm{\ell}_{\bm{A}} = \lowerh(
\bm{u}_{1},\bm{\ell}_{1})$, $\bm{u}_{A} = \upperh(\bm{u}_{1},\bm{\ell}_{1})$ similarly as in Section~\ref{sec:cut_general1}, \eqref{multi_dz_subproblem} is relaxed to a decomposable problem 
\begin{equation}
    \begin{aligned}
        \sum_{i=1}^{d_x} &\ \min_{\sigma_i \in \mathbb{R}} && q_i(x_{2,i} - \sigma_i)^2\\
        &\ \text{ s.t.} &&\ell (\bm{\chi})_i \leq \sigma_i \leq u(\bm{\chi})_i
    \end{aligned}
    \label{Form:multi relaxed sub prob}
\end{equation}
where $q_i := \diag(\bm{Q}_2)_i$ for $i \in [d_x]$, $\lambda = \mathbbm{1}_{|K_1|}^\top \bm{z}\kf$, and 
\begin{align*}
    \bm{\ell} (\bm{\chi}) &:= \max \bigg\{\bm{\ell}_{\bm{A}}+\frac{\bm{B}\js \bm{y}\js + \bm{C}\ks \bm{z}\ks}{1-\lambda},\ \bm{\ell}_{2},\\ 
    & \hspace{1.3cm}\frac{\bm{A}\bm{x}_{1} - \lambda \bm{u}_{\bm{A}}}{1-\lambda }+\frac{\bm{B}\js \bm{y}\js + \bm{C}\ks \bm{z}\ks}{1-\lambda }, \ \frac{\bm{x}_2 - \lambda \bm{u}_{2}}{1-\lambda} \bigg\},\\
    \bm{u}(\bm{\chi}) &:= \min\bigg\{\bm{u}_{\bm{A}}+\frac{\bm{B}\js \bm{y}\js + \bm{C}\ks \bm{z}\ks}{1-\lambda},\ \bm{u}_{2}, \\ 
    &\hspace{1.3cm}\frac{\bm{A}\bm{x}_{1} - \lambda \bm{\ell}_{\bm{A}}}{1-\lambda}+\frac{\bm{B}\js \bm{y}\js + \bm{C}\ks \bm{z}\ks}{1-\lambda},\  \frac{\bm{x}_2 - \lambda \bm{\ell}_{2}}{1-\lambda} \bigg\}.
\end{align*}
As in the $d_z=1$ case, we derive linear feasibility cuts from these bounds of the projection problem  \eqref{Form:multi relaxed sub prob}.

\begin{proposition}\label{Prop3}
For any feasible solution of \eqref{multi-dz 2-way}, the following eight \textit{$(K_1,K_2)$--feasibility cuts} hold: 
\begin{enumerate}[label=(\Alph*)]
    \item \label{itm:M1} $(1-\lambda) \bm{\ell}_{\bm{A}} + \bm{B}\js \bm{y}\js + \bm{C}\ks \bm{z}\ks \leq (1-\lambda) \bm{u}_{2}$
    \item \label{itm:M2} $(1-\lambda) \bm{\ell}_{2} \leq (1-\lambda) \bm{u}_{\bm{A}} + \bm{B}\js \bm{y}\js + \bm{C}\ks \bm{z}\ks$
    \item \label{itm:M3} $(1-\lambda) \bm{\ell}_{\bm{A}} +\bm{B}\js \bm{y}\js + \bm{C}\ks \bm{z}\ks \leq \bm{x}_2 - \lambda \bm{\ell}_{2}$
    \item \label{itm:M4} $\bm{x}_2 -\bm{B}\js \bm{y}\js -\bm{C}\ks \bm{z}\ks \leq \lambda \bm{u}_{2} + (1-\lambda) \bm{u}_{\bm{A}}$
    \item \label{itm:M5} $\bm{A}\bm{x}_{1} + \bm{B}\js \bm{y}\js + \bm{C}\ks \bm{z}\ks \geq \lambda \bm{\ell}_{\bm{A}} + (1-\lambda) \bm{\ell}_{2}$ 
    \item \label{itm:M6} $\bm{A}\bm{x}_{1} + \bm{B}\js \bm{y}\js + \bm{C}\ks \bm{z}\ks  \leq \lambda \bm{u}_{\bm{A}} + (1-\lambda) \bm{u}_{2}$
    \item \label{itm:M7} $\lambda \bm{u}_{2} \geq \bm{B}\jf \bm{y}\jf + \bm{C}\kf \bm{z}\kf + \lambda \bm{\ell}_{\bm{A}}$
    \item \label{itm:M8} $\bm{B}\jf \bm{y}\jf + \bm{C}\kf \bm{z}\kf + \lambda \bm{u}_{\bm{A}} \geq \lambda \bm{\ell}_{2}$
\end{enumerate}
Moreover, for any solution of (R\ref{multi-dz 2-way}) with \ref{itm:M1} - \ref{itm:M8}, corresponding \eqref{Form:multi relaxed sub prob} is feasible. (Proof in Appendix \ref{AppendixC}.)
\end{proposition}

\begin{corollary}\label{Prop4}
A quadratic cut valid for \eqref{multi-dz 2-way} can be generated at any solution $\bar{\bm{\chi}} = (\bar{x}_1, \bar{x}_2, \bar{y}\jf, \bar{y}\js, \bar{z}\kf, \bar{z}\ks)$ of (R\ref{multi-dz 2-way}) with \ref{itm:M1} - \ref{itm:M8} as in Table \ref{table:feasible_cut_multi}, where $\bar{\lambda}:= \mathbbm{1}_{|K_1|}^\top \bar{\bm{z}}\kf$, $\bar{\bm{\ell}}:= \bm{\ell} (\bar{\bm{\chi}})$ and $\bar{\bm{u}}:= \bm{u}(\bar{\bm{\chi}})$.
\end{corollary}

\begin{table*}[ht]
\caption{Quadratic cuts for multi-dimensional binary variable $d_z\geq 2$.}
\centering
\input{tables/Feasible_cuts_multi_dz_arXiv}
\captionsetup{justification=centering}
\label{table:feasible_cut_multi}
\end{table*}

Note that the cut generation with two-way disjunction can be applied for any nonempty set $K_1 \subseteq K$.

\subsection{Multi-period HCP}\label{sec:cut_general4}

The cut-generation process discussed in Section \ref{sec:cut_simple} - \ref{sec:cut_general3} considers a single-period HCP. However,
the proposed cuts can be applied to the multi-period HCP by simply generating them utilizing the constraints for each period $t$, $t \in [n]$, independently: 
\begin{equation*}
    \begin{aligned}
    & \bm{x}_{t+1}^{\top} \bm{Q}_{t+1} \bm{x}_{t+1} + \bm{y}_{t}^{\top}\bm{R}_{t} \bm{y}_{t} + \bm{z}_t^\top \bm{S}_t \bm{z}_t  \leq w_t\\
    & \bm{x}_{t+1}= \bm{A}_t \bm{x}_t + \bm{B}_t \bm{y}_{t}+ \bm{C}_t \bm{z}_t + \bm{f}_t,\\
    & \bm{G}_{t} \bm{z}_t \leq \bm{y}_{t} \leq \bm{H}_{t} \bm{z}_t,\\
    & \bm{\ell}_{t} \leq \bm{x}_{t} \leq \bm{u}_{t}, \ \bm{\ell}_{t+1} \leq \bm{x}_{t+1} \leq \bm{u}_{t+1}, \\ 
    & \bm{z}_t \in \{0,1\}^{d_z}, \ \bm{x}_t,\bm{x}_{t+1} \in \mathbb{R}^{d_x}, \ \bm{y}_{t} \in \mathbb{R}^{d_y}, w_t \in \mathbb{R}.
    \end{aligned}
\end{equation*}
This approach generates cuts from $n$ single-period problems and ignores the dependencies between consecutive periods. Therefore, it may have a larger relaxation gap compared to approaches considering multiple periods concurrently. However, the convexification in the original space keeps the number of variables small and may result in shorter computational times compared to methods that account for the interaction between periods. Details are discussed in Section \ref{sec:experiment}.


\section{Computational experiments}\label{sec:experiment}
In this section, we present computational experiments conducted to test the effectiveness of the proposed With-Cuts (WC) model. The experiments were performed on a 3.6-GHz processor, 32GB memory Linux machine using Python 3.8 and Drake \citep{drake} with embedded Mosek 9.0 solver. Presolve and automatic cut generation were enabled, and the node and variable selections for the B$\And$B were set to the default of Mosek with a time limit of one hour per instance. Four models are tested:
\begin{enumerate}
    \item MIQP: Solving \eqref{ORIG_PROB} directly.
    \item WC-NL: The With-Cuts (WC) model incorporating nonlinear cuts as in \eqref{1dim_nonlinear_ext_cut}.
    \item WC-G: WC model with gradient cuts \eqref{feas_cut_grad_ver}. Gradient cuts are added if $\mu_t^{new} \geq (1+10^{-6}) \mu_t^{prev} + 10^{-6}.$
    \item SPP: The generalized SPP approach of \cite{marcucci2024shortest}.
\end{enumerate}
The With-Cuts model strengthens the HCP formulations one period at a time, whereas the state-of-the-art generalized SPP approach \citep{marcucci2024shortest} utilizes $(d_z +1)^2$-way disjunctions considering the action space of two consecutive periods. Consequently, the SPP formulation is expected to be stronger than ours at the expense of a larger number of variables in the model. Our computational experiments reveal the tradeoff between the two approaches.

The synthetic dataset for the experiments is generated with the following specifications:
\begin{itemize}
    \item Problem dimensions: $n = 50$, $d_x \in \{1,\ldots,5\}$, $d_y \in \{d_x,\ldots, d_x+4\}$, $d_z=1$.
    \item Fixed cost matrices: $\bm{Q} = 2\bm{I}^{d_x}$, $\bm{R} = 0.01 \bm{I}^{d_y}$, $s=1$
    \item Fixed linear system dynamics: $\bm{A} = 1.5\bm{I}^{d_x} + \bm{M}^{d_x \times d_x}_{[0,1]}$, $\bm{B} = \bm{M}^{d_x\times d_y}_{[0,1]}$, $\bm{C} = 0.5\cdot\mathbbm{1}^{d_x}$, $\bm{f} = 0$ 
    \item Fixed variable bounds: $\bm{g}, -\bm{h} = 2.3\cdot \mathbbm{1}^{d_y}$, $\bm{\ell} = 0.1\cdot \mathbbm{1}^{d_x}$, $\bm{u} = 10\cdot \mathbbm{1}^{d_x}$
\end{itemize}
Here, $\bm{M}^{d_1\times d_2}_{[\alpha, \beta]}$ denotes a real-valued $d_1 \times d_2$  matrix with elements uniformly sampled from $[\alpha,\beta]$, and the superscripts of $\mathbbm{1}$ and $\bm{I}$ indicate their dimension. For each ($d_x, d_y$), ten instances are generated, and the average results across these instances are reported.

\subsection{Experimental results} \label{sec:experiment_result}
In Table \ref{tab:Mosek_N50_dz=1_repeat10_4model}, we report the relaxation gap and computational time for each model. Colored cells are used to highlight the results. Yellow cells indicate the models with the smallest relaxation gap, and blue cells indicate the models with the shortest computational time. The performance is highly dependent on the dimension of the state variables: $d_x$. SPP is superior for instances with $d_x \leq 3$, and WC-G outperforms others for higher dimensions ($d_x \geq 4$). Dependence on the dimension of the control variables ($d_y$) does not exhibit a clear trend. 
\begin{table*}[ht]
\caption{Experimental Results on Synthetic Data}
\centering
\input{tables/N=50_4models_comparison_arXiv}
\label{tab:Mosek_N50_dz=1_repeat10_4model}
\end{table*}

Figures \ref{fig:rgap_boxplot} and \ref{fig:stime_boxplot} present boxplots of the relaxation gap and the computational time as a function of $d_x$.
\begin{figure}[h!]
    \centering
    \resizebox{0.9\columnwidth}{!}{\includegraphics[width=\textwidth]{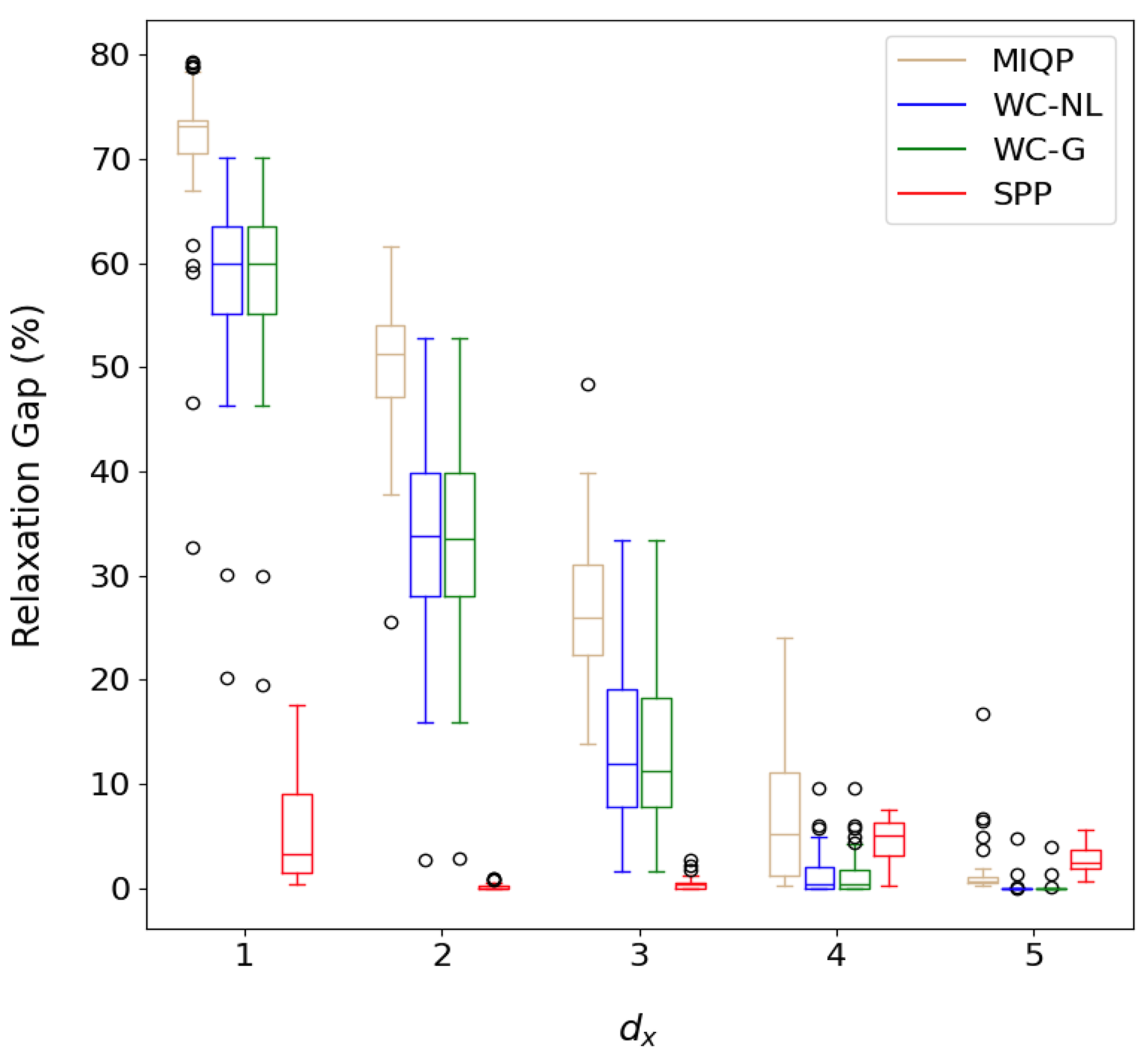}} 
    \vspace{-1mm}
    \caption{Relaxation gap as a function of $d_x$.}
    \label{fig:rgap_boxplot}
    \vspace{3mm}
    \resizebox{0.9\columnwidth}{!}{\includegraphics[width=\textwidth]{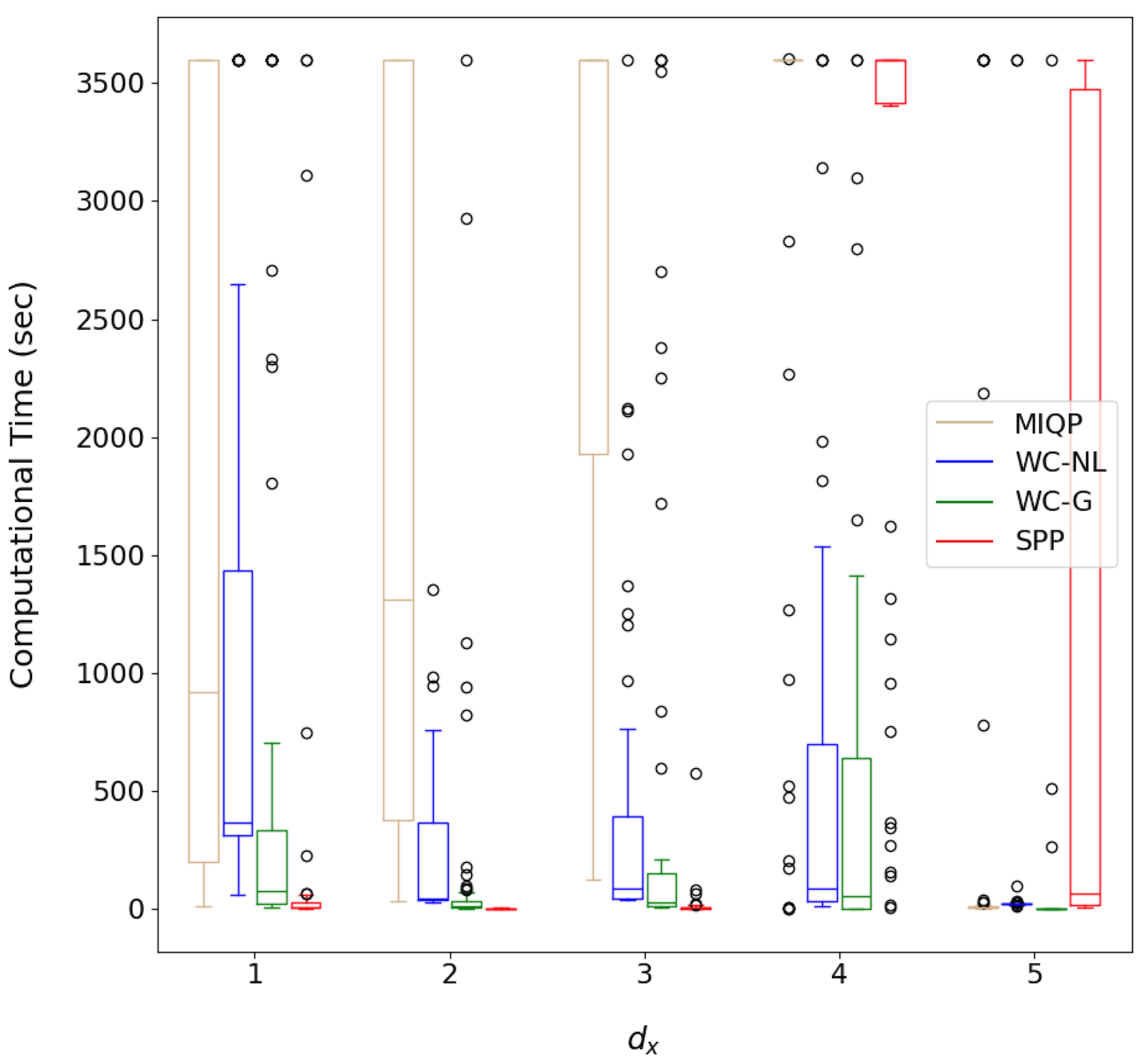}} 
    \vspace{-1mm}
    \caption{Computational time as a function of $d_x$.}
    \label{fig:stime_boxplot}
\end{figure}
MIQP exhibits consistently large relaxation gaps, particularly for smaller $d_x$ ($d_x \le 3$). Both WC models achieve significantly smaller relaxation gaps compared to MIQP, demonstrating the effectiveness of the proposed cut-generation process. For $d_x \geq 4$, WC models achieve very small relaxation gaps, while gaps remain large when $d_x \leq 3$. The SPP model consistently achieves a small relaxation gap for all $d_x$, as expected.

MIQP is computationally expensive in most cases. For $d_x\leq 3$, the SPP model is solved much faster than other models. Conversely, for $d_x \geq 4$, the WC-G has the shortest computational time in most cases. In particular, the computational times of MIQP and SPP models are very large compared to the WC models when $d_x =4$. Even when the relaxation gap is small, SPP takes considerable time to solve for $d_x \geq 4$ due to the large number of binary variables. Neither of the two versions of the WC models outperforms the other consistently, but WC-G with gradient cuts has a shorter computational time on average in most cases.


We further examine the performance of WC-G and SPP models as a function of $n$ for $d_x =2$ and $d_x=4$ with a time limit of one hour. The results are summarized in Figures \ref{fig:n_scale_dx2_boxplot} and \ref{fig:n_scale_dx4_boxplot}. For $d_x =2$, the relaxation gap and computational time for SPP remain relatively stable as $n$ grows. In contrast, the relaxation gap of WC-G increases as $n$ increases. The computational time of WC-G remains relatively stable for $n \leq 30$, but it fails to solve a few instances within one hour for $n \geq 40$. Conversely, when $d_x=4$, the trend is reversed. The relaxation gap and the computational time for WC-G remain small even for large $n$, while those of SPP increase rapidly as $n$ grows, resulting in many unsolved instances within one hour. This outcome underscores the significant impact of $d_x$ on the performance of the models.

\begin{figure*}[h!]
\centering
\subfloat[Relaxation gap\label{fig:rgap_dx2_boxplot}]{%
       \includegraphics[width=0.44\linewidth]{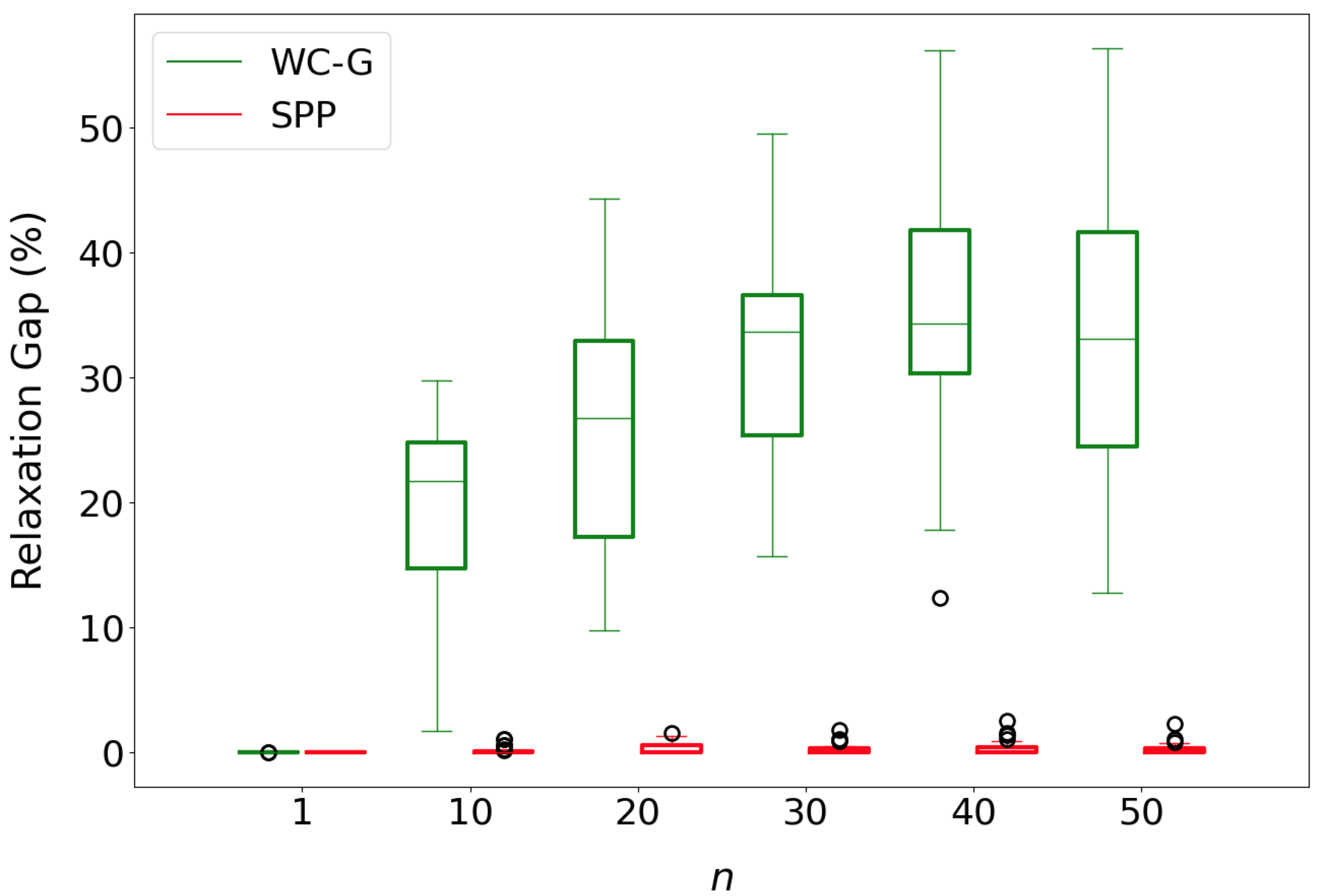}}
    \hspace{7mm}
\subfloat[Computational time\label{fig:stime_dx2_boxplot}]{%
       \includegraphics[width=0.453\linewidth]{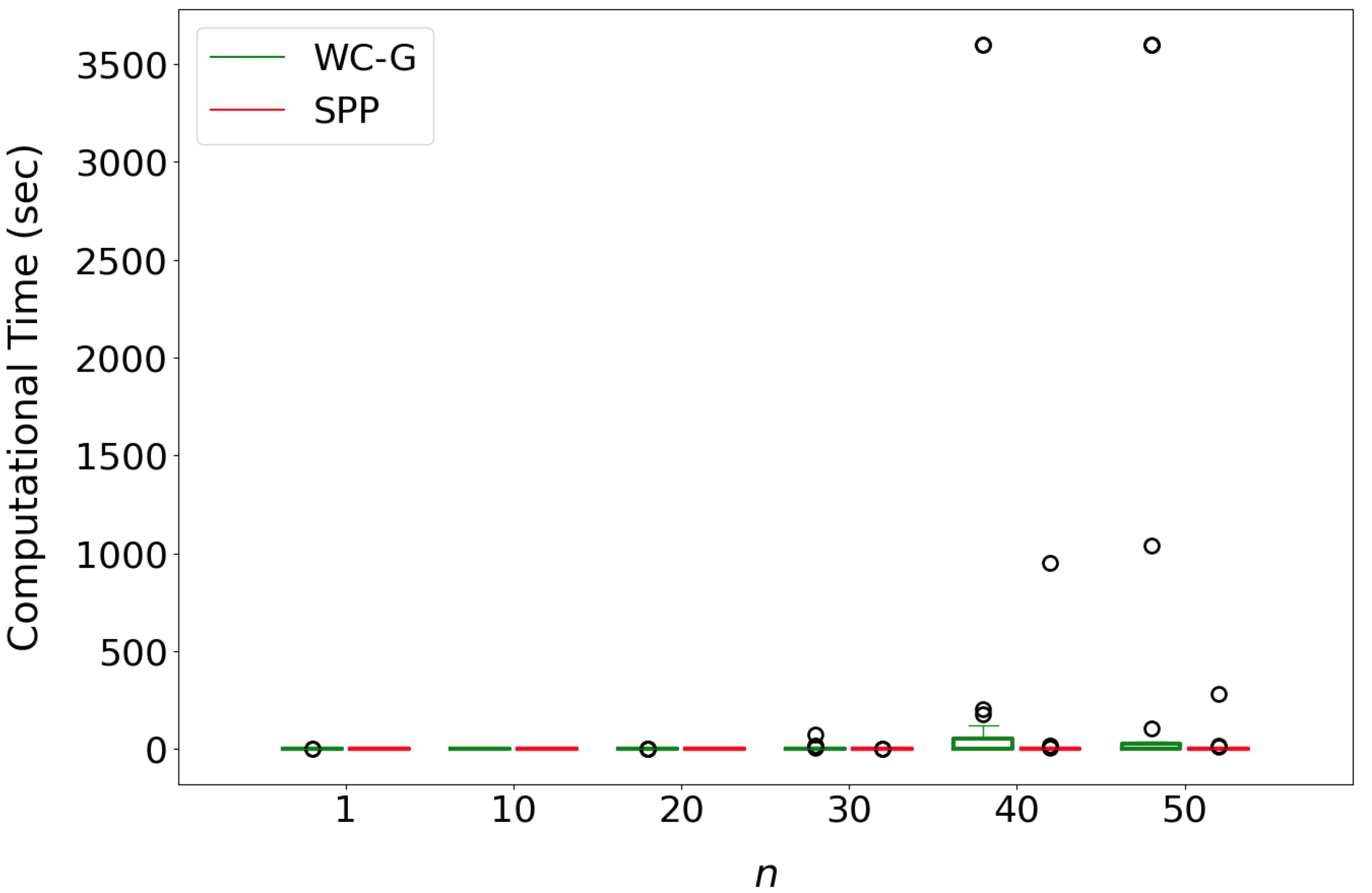}}
       \vspace{-0.5mm}
\caption{Performance of WC-G and SPP as a function of $n$ ($d_x=2$).}
\label{fig:n_scale_dx2_boxplot}
\vspace{2mm}
\subfloat[Relaxation gap\label{fig:rgap_dx4_boxplot}]{%
       \includegraphics[width=0.44\linewidth]{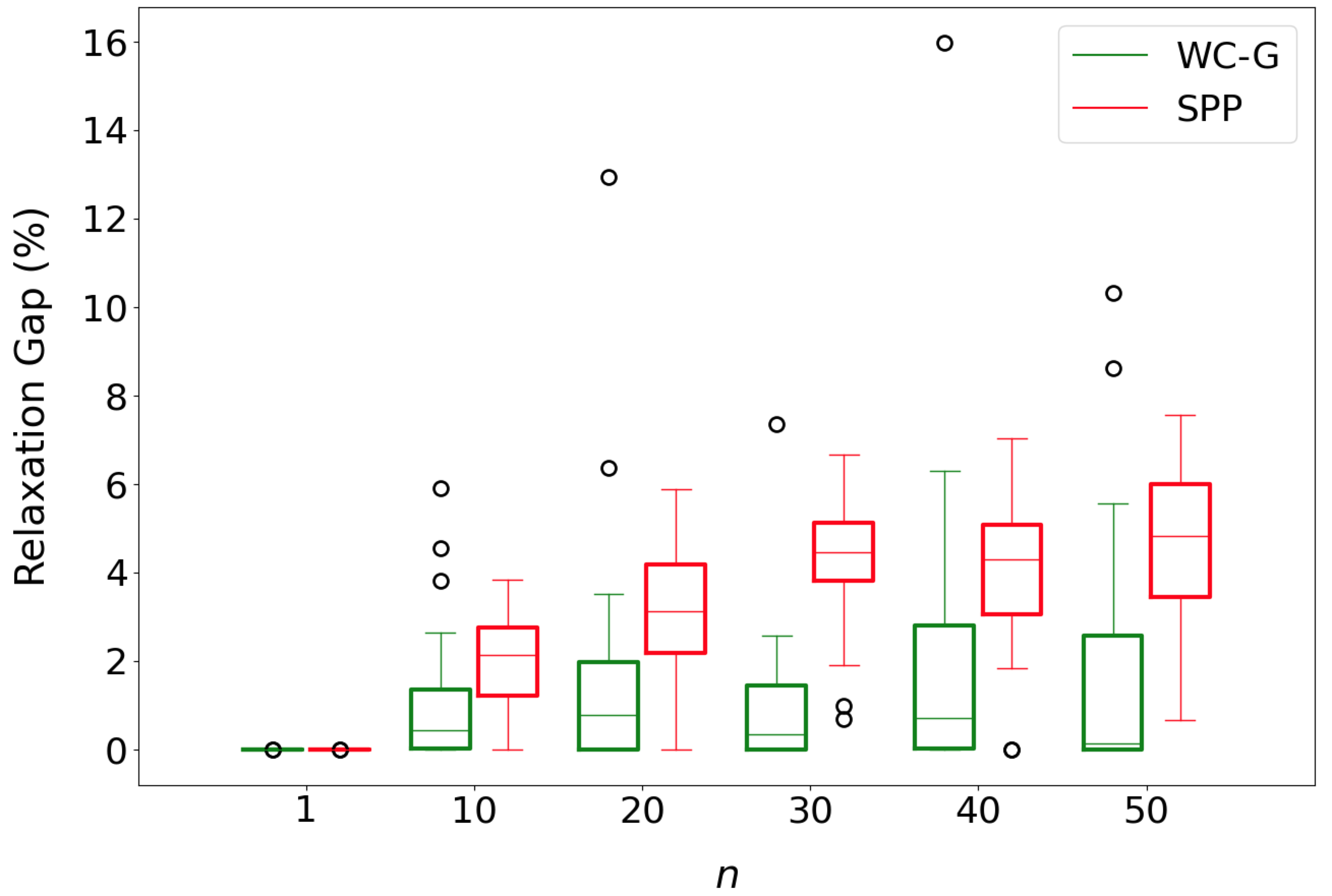}}
    \hspace{7mm}
\subfloat[Computational time\label{fig:stime_dx4_boxplot}]{%
       \includegraphics[width=0.453\linewidth]{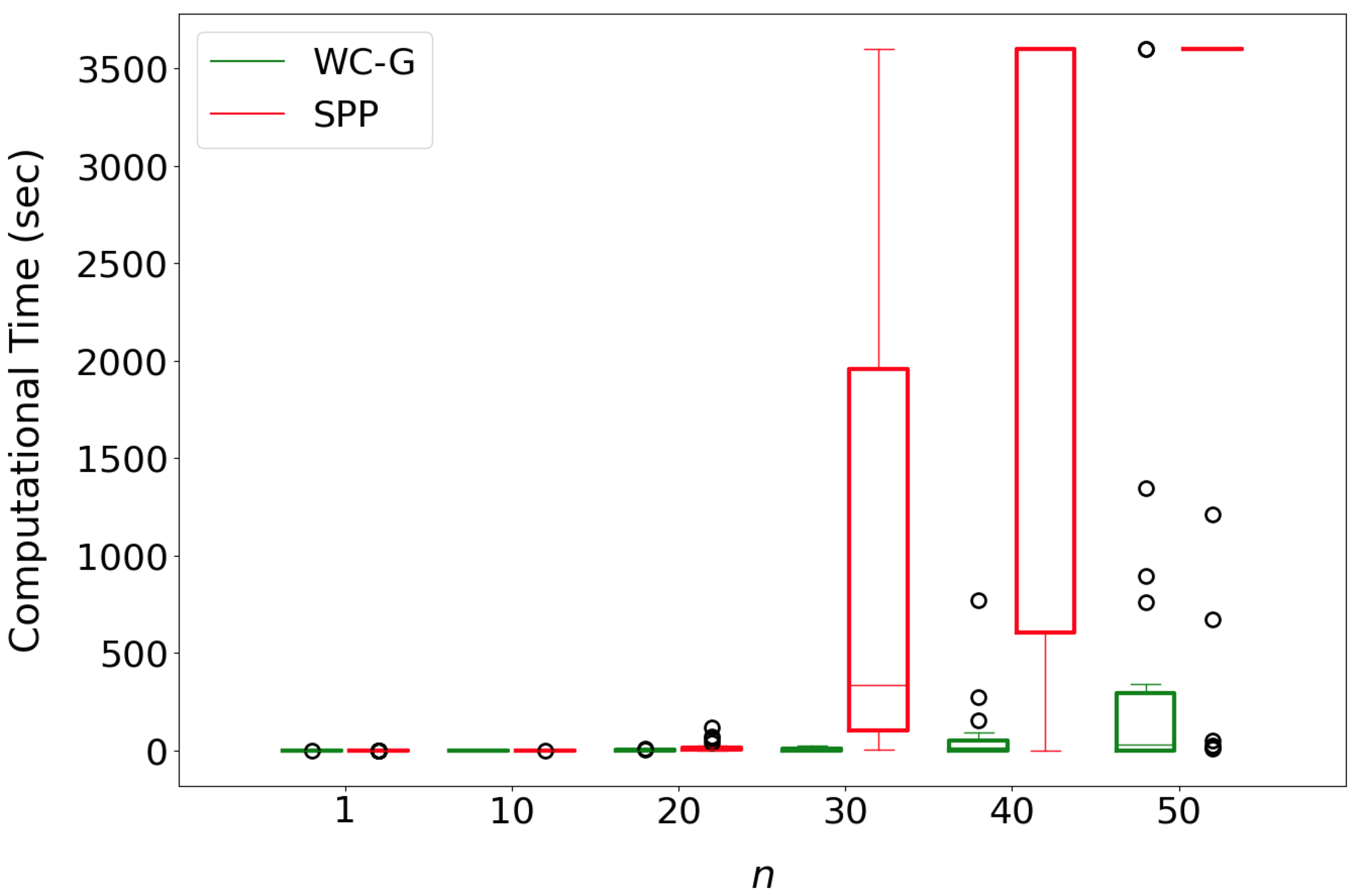}}
       \vspace{-0.5mm}
\caption{Performance of WC-G and SPP as a function of $n$ ($d_x=4$).}
\label{fig:n_scale_dx4_boxplot}
\end{figure*}

\subsection{Discussion: WC-G and SPP models} \label{sec:experiment_disc}

In this section, we discuss the factors affecting the performance of the two models, WC-G and SPP. The WC-G model applies the disjunction method independently to each period. For instance, when $d_z=1$, the feasible region for each period is partitioned into two sets, with cuts generated based on this partitioning. This localized approach enables the WC-G model to maintain simplicity by operating entirely in the original variable space, avoiding the introduction of additional binary variables, and minimizing computational overhead.
On the other hand, the SPP model can be interpreted as partitioning the feasible region into disjunctions, considering two consecutive periods simultaneously.  While this augmentation strengthens the relaxation further, it increases the number of binary variables and may lead to longer computational times, especially for large instances.
Table \ref{tab:WC_SPP_dimension} provides the dimensions of the two models. The WC-G model is formulated in the original variable space, whereas the SPP model is defined in an extended space.

Furthermore, the feasibility and gradient cuts generated in the WC-G model are all linear, whereas the SPP model employs nonlinear perspective functions directly, making the formulation inherently stronger but more challenging to solve.  These factors constitute the primary determinants of the models' performance.


\begin{table}[ht]
\caption{Variable size for models WC-G and SPP.}
\label{tab:WC_SPP_dimension}
\centering
\input{tables/WC_SPP_dimensions_arXiv}
\end{table}

\section{Application: energy management of power-split hybrid electric vehicle} \label{sec:HEV}

In this section, we present a numerical study on the energy management of a power-split hybrid electric vehicle (PS-HEV) introduced in \cite{borhan2011mpc}. The main objective is to efficiently control the PS-HEV's powertrain components to track a predefined reference cycle closely.

The powertrain of a PS-HEV consists of three main components: an internal combustion engine, an electric motor, and a generator. Their interaction is governed by a control system that coordinates the energy distribution to meet the performance targets. To model such a system, we define the state $x$, control $y$, and measured disturbance $v$ as follows:
\begin{equation}
    \bm{x} = \begin{bmatrix}
        \text{SOC} \\ \dot{m}_f 
    \end{bmatrix} \quad \bm{y} = \begin{bmatrix}
        V \\ w_{\text{eng}} \\ T_{\text{eng}}
    \end{bmatrix} \quad \bm{v} = \begin{bmatrix}
        V_{r} \\ T_{d} 
    \end{bmatrix}. \label{HEV_variables}
\end{equation}
\pagebreak
\noindent
The state variables comprise the battery's state of charge SOC ([0-1]) and the fuel consumption rate $\dot{m}_f$ (kg/sec). The control variables consist of the vehicle speed $V$ (m/s), the speed $w_{\text{eng}}$ (rad/sec) and the torque $T_{\text{eng}}$ (N$\cdot$m) of the engine. The reference speed $V_{r}$ (m/s) and driver's torque demand $T_{d}$ (N$\cdot$m) are given parameters. 

\subsection{Nonlinear MPC} \label{sec:nonlinearMPC}

In the model predictive control (MPC) approach, 
the model is discretized with a sampling time $T_s$, and $n$-period HCPs \eqref{nonlinearMPC} are solved iteratively. 
\begin{table}[ht]
    \begin{minipage}{\linewidth}
    \normalsize
\begin{subequations}\label{nonlinearMPC}
\begin{align}
    \min \ & \sum_{t = t_0}^{t_0+n} \Big[q_1 ( \text{SOC}(t) - \text{SOC}_{r} )^2 + q_2 (\dot{m}_f(t) )^2 \Big]  \nonumber\\
    & \ + \sum_{t = t_0}^{t_0+n-1} \Big[ r_1 ( V(t) - V_{r}(t))^2+ r_2 (w_{\text{eng}} (t) )^2  \nonumber\\
    & \hspace{2.3cm}{ + r_3 (T_{\text{eng}} (t))^2 + s z_{\text{eng}} (t) \Big]} \label{nonlinearMPC_obj}\\
    \text{s.t. } \ & \text{SOC}(t + 1) = \text{SOC}(t ) - \frac{ T_s}{C_{\text{batt}}}I(t) \label{nonlinearMPC_soc}\\
    &\dot{m}_f(t + 1) = \gamma \dot{m}_f(t) + \varphi_{m_f} (w_{\text{eng}} (t), T_{\text{eng}} (t)) \label{nonlinearMPC_mf}\\
    & I(t) = \frac{V_{\text{oc}} (t) - \sqrt{V_{\text{oc}} (t) ^2 - 4 R_{\text{batt}} (t) P_{\text{batt}} (t)}}{2 R_{\text{batt}}(t)} \label{nonlinearMPC_I}\\
    &P_{\text{batt}} (t) = \varphi_{\text{mot}}(w_{\text{mot}} (t), T_{\text{mot}} (t))  \nonumber\\
    &\qquad \qquad \qquad + \varphi_{\text{gen}}(w_{\text{gen}}(t), T_{\text{gen}}(t)) \label{nonlinearMPC_Pbatt} \\
    &V_\text{oc} (t) = \varphi_{V_\text{oc}}(\text{SOC}(t)) \label{nonlinearMPC_Voc} \\
    &R_\text{batt} (t) = \varphi_R^{\text{DC}}(\text{SOC}(t)) \mathbbm{1}\{P_{\text{batt}} (t) \geq 0\}  \nonumber\\
    & \qquad \quad \quad + \varphi_R^{\text{C}}(\text{SOC}(t)) \mathbbm{1}\{P_{\text{batt}} (t) < 0\} \label{nonlinearMPC_Rbatt} \\
    &w_{\text{gen}}(t) = \frac{N_S+N_R}{N_S} w_{\text{eng}}(t) - \frac{N_R }{N_S }w_{\text{mot}}(t) \label{nonlinearMPC_wgen}\\
    &w_{\text{mot}}(t) = \frac{g_f}{r_w}V(t) \label{nonlinearMPC_wmot}\\
    &T_{\text{gen}}(t) = - \frac{N_S}{N_S + N_R} T_{\text{eng}}(t) \label{nonlinearMPC_Tgen} \\
    &T_{\text{mot}}(t) = \frac{T_{d}(t)-T_{\text{b}}}{g_f} - \frac{N_R}{N_S+N_R}T_{\text{eng}}(t)  \label{nonlinearMPC_Tmot} \\
    &\text{SOC}^{\min} \leq \text{SOC} (t) \leq \text{SOC}^{\max} \label{nonlinearMPC_soc_bd}\\
    &\dot{m}_f^{\min} \leq \dot{m}_f (t) \leq \dot{m}_f^{\max} \label{nonlinearMPC_mf_bd}\\
    &V^{\min} \leq V (t) \leq V^{\max} \label{nonlinearMPC_V_bd}\\
    &w_{\text{eng}}^{\min} z_{\text{eng}}(t) \leq w_{\text{eng}} (t) \leq w_{\text{eng}}^{\max} z_{\text{eng}}(t)  \label{nonlinearMPC_weng_bd}\\ 
    &T_{\text{eng}}^{\min} z_{\text{eng}}(t) \leq T_{\text{eng}} (t)\leq T_{\text{eng}}^{\max} z_{\text{eng}}(t)  \label{nonlinearMPC_Teng_bd}\\
    &z_{\text{eng}}(t) \in \{0,1\}, \qquad \text{$t \in [t_0,t_0+n-1]$}. \label{nonlinearMPC_z_binary}
\end{align}
\end{subequations}
\end{minipage}
\end{table}

The system dynamics at period $t_0$ are governed by \eqref{nonlinearMPC_soc}--\eqref{nonlinearMPC_Tmot}. 
The state variables SOC and $\dot{m}_f$, in period $t+1$ is determined by \eqref{nonlinearMPC_soc} and \eqref{nonlinearMPC_mf}, where $C_{\text{batt}}$ is the battery capacity, $\alpha_{\text{eng}}$ is a constant determined by the vehicle type, and $I$ is the current computed as in \eqref{nonlinearMPC_I}.
The power $P_{\text{batt}}$, open-circuit voltage $V_{oc}$ (V), and internal resistance $R_{\text{batt}}$ ($\Omega$) of the battery are estimated using empirical maps as in \eqref{nonlinearMPC_Pbatt}--\eqref{nonlinearMPC_Rbatt}.
Additionally, equations \eqref{nonlinearMPC_wgen}--\eqref{nonlinearMPC_Tmot} define the speed and torque of the motor and generator.
Furthermore, \eqref{nonlinearMPC_weng_bd} and \eqref{nonlinearMPC_Teng_bd} ensure that the engine speed and torque belong to $w_{\text{eng}} \in [w_{\text{eng}}^{\min}, w_{\text{eng}}^{\max}]$ and $T_{\text{eng}} \in [T_{\text{eng}}^{\min}, T_{\text{eng}}^{\max}]$ when engine is on, and otherwise, $T_{\text{eng}}=w_{\text{eng}}=0$, with the indicator $z_{\text{eng}}$ of engine on/off. Additional bound constraints \eqref{nonlinearMPC_soc_bd} - \eqref{nonlinearMPC_V_bd} on $\text{SOC}$, $\dot{m}_f$, $V$ are imposed. 
The objective is to minimize \eqref{nonlinearMPC_obj}, where $\text{SOC}_{r}$ is the desired $\text{SOC}$.

\subsection{Linearized MPC} \label{sec:linearMPC}
To mitigate the computational challenge of the nonlinear model \eqref{nonlinearMPC}, a linearized MPC is employed. By projecting out variables other than those in \eqref{HEV_variables} and utilizing the gradient approach, the state in period $t+1$ is expressed as linear functions of the state, control, and binary variables in period $t$ as shown in \eqref{eq:linear_mpc}. 

Although the fuel consumption rate $\dot{m}_f$ remains independent of its previous value in the nonlinear model \eqref{nonlinearMPC} ($\gamma = 0$), we examine its influence by varying $\gamma$ as $0$, $0.01$, and $0.02$. 
The proper selection of sampling time and initial values is critical, as the feasibility and quality of the solution heavily depend on these choices.

\begin{table*}[ht]
    \begin{minipage}{\linewidth}
    \normalsize
    \begin{equation}
    \begin{aligned}
        \text{SOC}(t + 1) &= \text{SOC} (t) -\frac{T_s}{C_{\text{batt}}}\Bigg[I(t_0) + \frac{\partial I(t_0)}{\partial P_{\text{batt}}} \frac{\partial P_{\text{batt}}(t_0)}{\partial w_{\text{eng}}} (w_{\text{eng}}(t) - w_{\text{eng}}(t_0)) + \frac{\partial I(t_0)}{\partial P_{\text{batt}}}\frac{\partial P_{\text{batt}}(t_0)}{\partial T_{\text{eng}}} (T_{\text{eng}}(t) - T_{\text{eng}}(t_0))\\
        &\hspace{-7.5mm} + \left(\frac{\partial I(t_0)}{\partial V_{\text{oc}}}\frac{d V_{\text{oc}}(t_0)}{d \text{SOC}} + \frac{\partial I(t_0)}{\partial R_{\text{batt}}}\frac{d R_{\text{batt}}(t_0)}{d \text{SOC}}\right) (\text{SOC}(t) - \text{SOC} (t_0)) + \frac{\partial I(t_0)}{\partial V} (V(t) - V(t_0)) + \frac{\partial I(t_0)}{\partial T_{d}} (T_{d}(t) - T_{d}(t_0))\Bigg]\\
        \dot{m}_f(t + 1) &= \gamma \dot{m}_f(t) + \varphi_{m_f} (w_{\text{eng}} (t_0), T_{\text{eng}} (t_0)) + \frac{\partial \varphi_{m_f}(t_0)}{\partial w_{\text{eng}}} \left( w_{\text{eng}} (t) - w_{\text{eng}} (t_0) \right) + \frac{\partial \varphi_{m_f}(t_0)}{\partial T_{\text{eng}}} \left( T_{\text{eng}} (t) - T_{\text{eng}} (t_0) \right)
    \end{aligned}
    \label{eq:linear_mpc}
\end{equation}
\end{minipage}
\end{table*}


\subsection{Experimental results} \label{sec:HEV_result}
In our experiment, we employ a setup similar to the one in Section \ref{sec:experiment}, utilizing the Gurobi 9.0.2 solver with parameter values detailed in Appendix \ref{AppendixE}.
We assess the improvement provided by the proposed cuts to the WC-G model compared to the original MIQP formulation. The large number of conic quadratic constraints in the SPP model caused frequent numerical errors; therefore, 
it is excluded from this comparison.

Table \ref{tab:HEV_full_comparison} presents average and maximum values for the relaxation gap, computational time, and the number of nodes in B\&B for varying costs for the vehicle speed ($r_1 = 1, 10, 100$), fuel consumption rate dependencies to the previous state ($\gamma = 0, 0.01, 0.02$), and sampling times ($T_s=0.5, 1$ sec).
\begin{table*}[ht]
    \caption{Comparison of MIQP and WC-G for the PS-HEV Application.} 
    \label{tab:HEV_full_comparison}
    \resizebox{\textwidth}{!}{
    \centering
    \input{tables/HEV_result_arXiv}}
\end{table*}
For a fixed $r_1$ and sampling time $T_s = 1$, the relaxation gap of the MIQP remains nearly constant regardless of the value of $\gamma$. In contrast, the relaxation gap for the WC-G model increases with $\gamma$. Both models exhibit increased computational time and the number of nodes when $\gamma > 0$ compared to $\gamma=0$. However, no specific trend between $\gamma = 0.01$ and $\gamma=0.02$ is observed. 
Regarding $r_1$, the relaxation gap shows no discernable pattern, but the computational time and the number of B\&B nodes increase notably for both models as $r_1$ increases. 

Similar trends are observed for sampling time $T_s = 0.5$. However, with more frequent sampling, the size of both models is twice as large, requiring more computational effort. The number of B\&B nodes and the time required to solve both models are more than doubled compared to $T_s = 1$.

The proposed WC-G model consistently outperforms the original MIQP in terms of the relaxation gap, the number of B\&B nodes, and the required computational time. Overall, the WC-G model leads to about 10x speed up on average. It also is significantly more robust compared to the MIQP model with much lower run times for the worst cases. Unlike for MIQP, the average run times for WC-G are well within the sampling time. Moreover, utilizing linear cuts in the original space improves the relaxations yet maintains QP subproblems that are numerically robust.

\section{Conclusion}\label{sec:conclusion}
In this paper, we study an $n$-period hybrid control problem (HCP) formulated as a mixed-integer quadratic programming problem with linear system dynamics. Utilizing disjunctive programming and projections, we derive the convex hull representation of the epigraph set of the single-period HCP in an extended space. Then, we further project out auxiliary variables to generate two types of cuts in the original variable space, the \textit{feasibility} and \textit{gradient cuts}, which improve the perspective cut by employing the linear dynamics constraint. We show the effectiveness of the proposed cut-generation process by comparing the strengthened model with the original MIQP formulation and the state-of-the-art generalized SPP model in the literature. Additionally, we apply our approach to the energy management of a power-split hybrid electric vehicle, demonstrating improved model performance through experiments.

\section*{Appendix}

\subsection{Proof of Proposition \ref{Prop1}}\label{AppendixA}
First consider \ref{itm:1} and \ref{itm:2}. Suppose \ref{itm:1} is violated by some $\Tilde{v} = \left(\Tilde{x}_1, \Tilde{x}_2, \Tilde{y}, \Tilde{z}, \Tilde{w}\right) \in \myZ_1$, $u_{2} < \ell_{a}$. Then,
\begin{align*}
    &u_{2} \geq \Tilde{x}_2 = a\Tilde{x}_1 +b\Tilde{y}+c\Tilde{z} \geq \ell_{a} + b\Tilde{y} + c\Tilde{z} > u_{2}+b\Tilde{y}+c\Tilde{z}\\
    &\Rightarrow \ \Tilde{z}=1, \ b\Tilde{y}+c <0.
\end{align*}
Similarly, if \ref{itm:2} is violated, $\ell_{2} > u_{a}$, 
\begin{align*}
    &\ell_{2} \leq \Tilde{x}_2 =a\Tilde{x}_1+b\Tilde{y}+c\Tilde{z} \leq u_{a}+b\Tilde{y}+c\Tilde{z} < \ell_{2} + b\Tilde{y}+c\Tilde{z} \\ 
    &\Rightarrow \ \Tilde{z}=1, \ b\Tilde{y}+c >0.
\end{align*}
Thus, if either \ref{itm:1} or \ref{itm:2} is violated, $z = 1$ holds for any feasible solution of (\ref{SIMPLE_PROB}). 

We now prove that \ref{itm:3} - \ref{itm:8} hold for any $\Tilde{v} \in \myZ_1$ by contradiction. If $\Tilde{v}$ violates \ref{itm:3}, $\Tilde{x}_2 - \ell_{2} \Tilde{z} < \ell_{a}(1-\Tilde{z})$,
\begin{align*}
    &\ell_{a}(1-\Tilde{z}) + \ell_{2}\Tilde{z} > \Tilde{x}_2 \geq \ell_{2} \ \Rightarrow \ (\ell_{a} - \ell_{2})(1-\Tilde{z}) > 0\\
     & \Rightarrow \ \Tilde{z}=0, \ \ell_{a} > \Tilde{x}_2 = a\Tilde{x}_1.
\end{align*}
However, this cannot occur since $a\Tilde{x}_1 \geq \ell_{a}$. Thus, \ref{itm:3} cannot be violated by any solution of (\ref{SIMPLE_PROB}) and is a valid cut for $\conv(\myZ_1)$.
In the same manner, when \ref{itm:4} is violated by $\Tilde{v} \in \myZ_1$, $\Tilde{x}_2 - u_{2}\Tilde{z} > u_{a}(1-\Tilde{z})$,
\begin{align*}
    &u_{a}(1-\Tilde{z}) + u_{2}\Tilde{z} < \Tilde{x}_2 \leq u_{2} \ \Rightarrow \ (u_{a} - u_{2})(1-\Tilde{z}) < 0\\ 
    &\Rightarrow \ \Tilde{z}=0 \ u_{a} < \Tilde{x}_2 = a\Tilde{x}_1,
\end{align*}
which contradicts $a\Tilde{x}_1 \leq u_{a}$. Thus, \ref{itm:4} is valid for \eqref{SIMPLE_PROB}. It also can be shown that \ref{itm:5} and \ref{itm:6} cannot be violated by any $\Tilde{v} \in \myZ_1$ as follows: Negate \ref{itm:5} and \ref{itm:6}
\begin{align*}
    \ref{itm:5} \ &\ell_{a}\Tilde{z}+\ell_{2}(1-\Tilde{z}) > \Tilde{x}_2 - b\Tilde{y}-c\Tilde{z} = a \Tilde{x}_1 \geq \ell_a \\ 
    &\Rightarrow \ 0 <(\ell_{2} -\ell_{a})(1-\Tilde{z}) \ \Rightarrow \ \Tilde{z}=0, \ \Tilde{x}_2 = a\Tilde{x}_1 < \ell_{2} \\ 
    &\rightarrow \ \text{Contradicts } \Tilde{x}_2 \geq \ell_{2}.\\
    \ref{itm:6} \ &u_{a}\Tilde{z} +u_{2}(1-\Tilde{z}) < \Tilde{x}_2 - b\Tilde{y}-c\Tilde{z} = a\Tilde{x}_1 \leq u_{a} \\ 
    &\Rightarrow \ (u_{a} - u_{2})(1-\Tilde{z}) >0 \ \Rightarrow \ \Tilde{z}=0, \ \Tilde{x}_2 = a\Tilde{x}_1 > u_{2} \\ 
    &\rightarrow \ \text{Contradicts } \Tilde{x}_2 \leq u_{2}.
\end{align*}
Therefore, \ref{itm:5} and \ref{itm:6} are valid cuts. Similarly, \ref{itm:7} and \ref{itm:8} are valid cuts for \eqref{SIMPLE_PROB}: Negate \ref{itm:7} and \ref{itm:8}
\begin{align*}
    \ref{itm:7} \ & b\Tilde{y}+c\Tilde{z} > (u_{2} - \ell_{a})\Tilde{z} \ \Rightarrow \ \Tilde{z}=1 \\
    & \Rightarrow \ \Tilde{x}_2 = a\Tilde{x}_1 + b\Tilde{y}+c > a\Tilde{x}_1 +u_{2} - \ell_{a} \geq u_{2}\\
    &\rightarrow \ \text{Contradicts } \Tilde{x}_2 \leq u_{2}.\\
    \ref{itm:8} \ & b\Tilde{y}+c\Tilde{z} < (\ell_{2} - u_{a})\Tilde{z} \ \Rightarrow \ \Tilde{z}=1 \\
    &\Rightarrow \ \Tilde{x}_2 = a\Tilde{x}_1 + b\Tilde{y}+c < a\Tilde{x}_1 +\ell_{2} - u_{a} \leq \ell_{2}\\
    &\rightarrow \ \text{Contradicts } \Tilde{x}_2 \geq \ell_{2}.
\end{align*}

\subsection{Proof of Corollary \ref{Prop2_revised}}\label{Appendix:nonlinear_proof}
When (\ref{MIN_PROB}) is feasible, the optimal solution $\sigma^*$ has one of the three values: $x_2$, $\ell (x_1,x_2, z)$, and $u(x_1,x_2,z)$. 

If $x_2 \in \left[\ell (x_1,x_2, z), u(x_1,x_2, z) \right]$, then $\sigma^* = x_2$ and $\tau(x_1, x_2, y, z) = 0$. This results in the nonlinear cut equivalent to the perspective cut 
\begin{align*}
    w \geq q_2 x_2^2 + z + \frac{ry^2}{z}.
\end{align*}

If $x_2 > u(x_1,x_2, z)$, the optimal solution of \eqref{MIN_PROB} is $\sigma^* = u(x_1,x_2, z)$ with $\tau(x_1, x_2, y, z) = \left(\bar{x}_2 - u(x_1,x_2,z)\right)^2$. Note that $u(x_1,x_2,z) = u_{a}$ or $u(x_1,x_2,z) = \frac{ax_1 - \ell_{a}z}{1-z}$ in this case, since $u_{2} \geq x_2 > u(x_1,x_2,z)$ and $\frac{x_2 - \ell_{2}z}{1-z} \geq x_2 > u(x_1,x_2,z)$ as $x_2 \geq \ell_{2}$. Therefore, a nonlinear cut is given as
\begin{align*}
    w \geq q_2 \left( \frac{1}{z} - 1\right) \left( x_2 - u(x_1,z) \right)^2 + q_2 x_2^2 + z + \frac{ry^2}{z}.
\end{align*}

For $x_2 < \ell (x_1,x_2, z)$, $\tau(x_1, x_2, y, z) = \left(\bar{x}_2 - \ell (x_1,x_2,z)\right)^2$ with the optimal solution $\sigma^* = \ell (x_1,x_2, z)$. Similar to the case of $x_2 > u(x_1,x_2,z)$, $x_2 < \ell (x_1,x_2,z)$ only occurs when $\ell (x_1,x_2,z)$ is either $\ell_{a}$ or $\frac{ax_1 -u_{a}z}{1-z}$. Thus, the resulting nonlinear cut is
\begin{align*}
    w \geq q_2 \left( \frac{1}{z} - 1\right) \left( \ell (x_1,z) - x_2 \right)^2 + q_2 x_2^2 + z + \frac{ry^2}{z}.
\end{align*}
\subsection{Proof of Proposition \ref{Prop3}}\label{AppendixC}
To satisfy $\bar{\bm{\ell}} \leq \bar{\bm{u}}$, \ref{itm:M1} - \ref{itm:M8} need to hold. We show that none of the constraints can be violated by a solution $\bm{v} = \left(\bm{x}_1, \bm{x}_2, \bm{y}\jf, \bm{y}\js, \bm{z}\kf, \bm{z}\ks \right) \in \myX_1$ by contradiction, and thus, they are valid cuts for (\ref{multi-dz 2-way}). Note $\lambda = \mathbbm{1}_{|K_1|}^\top \bm{z}\kf \in \{0,1\}$.

Suppose $\bm{v}$ violates \ref{itm:M1}, $(1-\lambda)\bm{\ell}_{\bm{A}} + \bm{B}\js \bm{y}\js + \bm{C}\ks \bm{z}\ks > (1-\lambda)\bm{u}_{2}$. If $\lambda = 1$, both sides are $0$. Thus, $\lambda = 0$ and $\bm{z}\kf = 0$, $\bm{y}\jf = 0$, and
\begin{equation*}
    \begin{aligned}
        \bm{u}_{2} & \geq \bm{x}_2 = \bm{A}\bm{x}_{1} + \bm{B}\js \bm{y}\js + \bm{C}\ks \bm{z}\ks \geq \ell_{\bm{A}} + \bm{B}\js \bm{y}\js + \bm{C}\ks \bm{z}\ks \\
        & > \lambda \bm{\ell}_{\bm{A}} + (1-\lambda) \bm{u}_{2}= \bm{u}_{2} \ \rightarrow \ \text{Contradiction}.
    \end{aligned}
\end{equation*}
Similarly, if \ref{itm:M2} is violated, $(1-\lambda) \bm{\ell}_{2} > (1-\lambda)\bm{u}_{\bm{A}}+ \bm{B}\js \bm{y}\js + \bm{C}\ks \bm{z}\ks$, $\lambda$ needs to be $0$. Thus, 
\begin{align*}
    \bm{\ell}_{2} &\leq \bm{x}_2 =\bm{A}\bm{x}_{1}+ \bm{B}\jf \bm{y}\jf+\bm{C}\kf \bm{z}\kf \leq \bm{u}_{\bm{A}}+ \bm{B}\jf \bm{y}\jf+\bm{C}\kf \bm{z}\kf \\
    & < \lambda \bm{u}_{\bm{A}} + (1-\lambda) \bm{\ell}_{2}= \bm{\ell}_{2} \ \rightarrow \ \text{Contradiction}.
\end{align*}
Constraints \ref{itm:M3} - \ref{itm:M8} can also be proved by negating each condition and showing contradiction for any $\bm{v} \in \myX_1$.

\subsection{Power-split hybrid electric vehicle}\label{AppendixE}
Here, we report the model parameters used in the PS-HEV model. The constant values are set as follows:
\begin{itemize}
    \item radius of sun gear $N_S = 30m$ and ring gear $N_R = 78m$
    \item friction brake torque $T_b = 0 N \cdot m$:
    \item final transmission gear ratio $g_f = 3.268$
    \item wheel radius $r_w = 0.32 m$
    \item battery capacity $C_{\text{batt}} = 24.7Ah$, target state of charge $\text{SOC}^{\text{ref}} = 0.55$
\end{itemize}
The measured disturbance $V_{r}$ and $T_{d}$ are designed using a real driving cycle data. The bounds and the initial values of the decision variables are set as 
\begin{itemize}
    \item state of charge $\text{SOC} \in [0.05, 0.55]$, $\text{SOC}^0 = 0.47$
    \item fuel consumption rate $\dot{m}_f \in [0, 45.45]$, $\dot{m}_f^0 = 0$ (kg/s)
    \item vehicle speed $V \in [0, 35.486]$, $V^0 = 9.924$ (m/s)
    \item engine speed $w_{\text{eng}} \in [80, 600]$, $w_{\text{eng}}^0 = 300$ (rad/s)
    \item engine torque $T_{\text{eng}} \in [0, 168]$, $T_{\text{eng}}^0 = 80$ (N$\cdot$m)
\end{itemize}
The cost function is defined as in \eqref{nonlinearMPC} with $q_1 = 5500$, $q_2 = 10$, $r_1=1, 10, 100$, $r_2 = r_3 = 0$, $s = 1000$. The linearized MPC setting is
\begin{itemize}
    \item sampling time: $T_s = 0.5, 1$ sec
    \item MPC time horizon: $20$ time periods
    \item simulation duration: $M= 100$ sec
\end{itemize}
\ignore{
and each problem is solved with Gurobi 9.0.2 solver with 
\begin{itemize}
    \item primal feasibility = 1e-9
    \item integer feasibility = 1e-9
    \item dual feasibility (optimality) = 1e-9
    \item PSD tolerance = 1e-12
    \item presolve, cut, heuristic: off
\end{itemize}
}

\section*{Acknowledgement}
We are grateful to Prof. Francesco Borrelli and Yongkeun Choi for providing us the MPC model data used in PS-HEV experiments.




\bibliographystyle{IEEEtran}
\bibliography{ref_arXiv}

\pagebreak

\vfill

\end{document}

%% file: def_arXiv.tex
\newcommand{\alert}[1]{\textbf{}}

\newcommand{\BEAS}{\begin{eqnarray*}}
\newcommand{\EEAS}{\end{eqnarray*}}
\newcommand{\BEA}{\begin{eqnarray}}
\newcommand{\EEA}{\end{eqnarray}}
\newcommand{\BEQ}{\begin{equation}}
\newcommand{\EEQ}{\end{equation}}
\newcommand{\BIT}{\begin{itemize}}
\newcommand{\EIT}{\end{itemize}}
\newcommand{\BNUM}{\begin{enumerate}}
\newcommand{\ENUM}{\end{enumerate}}
\newcommand{\BEL}[1]{\begin{equation}\label{#1}}
\newcommand{\EEL}{\end{equation}}

\newcommand{\BA}{\begin{array}}
\newcommand{\EA}{\end{array}}

\newtheorem{theorem}{Theorem}

\newtheorem{corollary}[theorem]{Corollary}
\newtheorem{proposition}[theorem]{Proposition}
\newtheorem{remark}[theorem]{Remark}

\newcounter{exno}



%% file: tables/Feasible_cuts_arXiv.tex
\begin{tabular}{c|c|c}
     \hline     \hline\textbf{$\bar{x}_2$}&\textbf{$\bar{\sigma}$}&Conditions\\
     \hline
     \multirow{3}{*}{$\bar{x}_2 < \bar{\ell}$}& \makecell{$\bar{\ell} = \ell_{a}$ $\star$} & \makecell{$\bar{x}_2 < \ell_{a}$,\\ $\bar{x}_2 \leq b\bar{y}+c\bar{z}+u_{a}\bar{z}+\ell_{a}(1-\bar{z})$}\\
     \cline{2-3}
     &\makecell{$\bar{\ell} = \frac{a\bar{x}_1 - u_{a}\bar{z}}{1-\bar{z}}$} & \makecell{$\bar{x}_2\bar{z} > b\bar{y}+c\bar{z}+u_{a}\bar{z}$,\\$\bar{x}_2 \geq b\bar{y}+c\bar{z}+u_{a}\bar{z} + \ell_{a}(1-\bar{z})$ $\star$}\\
     
     \hline
     \makecell{$\bar{\ell} \leq \bar{x}_2 \leq \bar{u}$ }&$\bar{x}_2$& \makecell{ $b\bar{y}+c\bar{z}+\ell_{a}\bar{z} \leq \bar{x}_2\bar{z} \leq b\bar{y}+c\bar{z}+u_{a}\bar{z}$,\\
     $\bar{x}_2 \leq u_{a}$ $\dagger$, $\bar{x}_2 \geq \ell_{a}$ $\star$}\\
     
     \hline
     \multirow{3}{*}{$\bar{x}_2 > \bar{u}$} &\makecell{$\bar{u} = u_{a}$ $\dagger$}& \makecell{$\bar{x}_2 >u_{a}$,\\ $\bar{x}_2 \geq b\bar{y}+ c\bar{z}+\ell_{a}\bar{z} + u_{a}(1-\bar{z})$}\\
     \cline{2-3}
     &\makecell{$\bar{u} = \frac{a\bar{x}_1 - \ell_{a}\bar{z}}{1-\bar{z}}$}&\makecell{$\bar{x}_2\bar{z} < b\bar{y}+c\bar{z}+\ell_{a}\bar{z}$, \\ $\bar{x}_2 \leq b\bar{y}+c\bar{z}+\ell_{a}\bar{z} + u_{a}(1-\bar{z})$ $\dagger$}\\
     \hline \hline
\end{tabular}

%% file: tables/Feasible_cuts_multi_dz_arXiv.tex
\begin{tabular}{c|c|c}
         \hline \hline
         $\bar{\bm{x}}_2$& $\bar{\bm{\sigma}}_i$ &Conditions\\
         
         \hline
         \multirow{3}{*}{$\bar{x}_{2,i} < \bar{\ell}_i$}& $\bar{\ell}_i = \ell_{\bm{A},i}+\frac{\left(\bm{B}\js \bar{\bm{y}}\js + \bm{C}\ks \bar{\bm{z}}\ks\right)_i}{1-\bar{\lambda}}$ & \makecell{\vspace{-2.5mm}\\ $(1-\bar{\lambda})\bar{x}_{2,i} < (1-\bar{\lambda})\ell_{\bm{A},i}+\left(\bm{B}\js \bar{\bm{y}}\js + \bm{C}\ks \bar{\bm{z}}\ks\right)_i$, \\[2pt] $(\bm{A} \bar{\bm{x}}_1)_i \leq \bar{\lambda} u_{\bm{A},i}+ (1-\bar{\lambda})\ell_{\bm{A},i}$\vspace{0.8mm}}\\
         \cline{2-3}
         &$\bar{\ell}_i = \frac{\left(\bm{A} \bar{\bm{x}}_1 - \bar{\lambda} \bm{u}_{\bm{A}} + \bm{B}\js \bar{\bm{y}}\js + \bm{C}\ks \bar{\bm{z}}\ks\right)_i}{1-\bar{\lambda}}$ & \makecell{\vspace{-2.5mm}\\ $\bar{\lambda} \bar{x}_{2,i} > \left(\bm{B}\jf \bar{\bm{y}}\jf + \bm{C}\kf \bar{\bm{z}}\kf\right)_i + \bar{\lambda} u_{\bm{A},i}$,\\[2pt] $(\bm{A} \bar{\bm{x}}_1)_i \geq \bar{\lambda} u_{\bm{A},i} + (1-\bar{\lambda})\ell_{\bm{A},i}$\vspace{0.8mm}}\\
         
         \hline
         $\bar{\ell}_i \leq \bar{x}_{2,i} \leq \bar{u}_i$& $\bar{x}_{2,i}$& \makecell{\vspace{-2.5mm}\\ $(1-\bar{\lambda})\ell_{\bm{A},i}+ \left(\bm{B}\js \bar{\bm{y}}\js + \bm{C}\ks \bar{\bm{z}}\ks\right)_i \leq (1-\bar{\lambda})\bar{x}_{2,i} \leq (1-\bar{\lambda})u_{\bm{A},i}+ \left(\bm{B}\js \bar{\bm{y}}\js + \bm{C}\ks \bar{\bm{z}}\ks\right)_i$,\\[2pt] 
         $\left(\bm{B}\jf \bar{\bm{y}}\jf + \bm{C}\kf \bar{\bm{z}}\kf\right)_i + \bar{\lambda} \ell_{\bm{A},i} \leq \bar{\lambda} \bar{x}_{2,i} \leq \left(\bm{B}\jf \bar{\bm{y}}\jf + \bm{C}\kf \bar{\bm{z}}\kf\right)_i+ \bar{\lambda} u_{\bm{A},i}$\vspace{0.8mm}}\\
         
         \hline
         \multirow{3}{*}{$\bar{x}_{2,i} > \bar{u}_i$}&$\bar{u}_i = u_{\bm{A},i}+\frac{\left(\bm{B}\js \bar{\bm{y}}\js + \bm{C}\ks \bar{\bm{z}}\ks\right)_i}{1-\bar{\lambda}}$& \makecell{\vspace{-2.5mm}\\ $(1-\bar{\lambda})u_{\bm{A},i}+ \left(\bm{B}\js \bar{\bm{y}}\js + \bm{C}\ks \bar{\bm{z}}\ks \right)_i < (1-\bar{\lambda})\bar{x}_{2.i}$, \\[2pt] $(\bm{A}\bar{\bm{x}}_1)_i \geq \bar{\lambda} \ell_{\bm{A},i} + (1-\bar{\lambda})u_{\bm{A},i}$\vspace{0.8mm}}\\
         \cline{2-3}
         &$\bar{u}_i = \frac{\left(\bm{A} \bar{\bm{x}}_1 - \bar{\lambda} \ell_{\bm{A}} + \bm{B}\js \bar{\bm{y}}\js + \bm{C}\ks \bar{\bm{z}}\ks\right)_i}{1-\bar{\lambda}}$&\makecell{\vspace{-2.5mm}\\ $\bar{\lambda} \bar{x}_{2,i} < \left(\bm{B}\jf \bar{\bm{y}}\jf + \bm{C}\kf \bar{\bm{z}}\kf\right)_i + \bar{\lambda} \ell_{\bm{A},i}$, \\[2pt]  $(\bm{A}\bar{\bm{x}}_1)_i \leq \bar{\lambda} \ell_{\bm{A},i} + (1-\bar{\lambda})u_{\bm{A},i}$\vspace{0.8mm}}\\
         \hline
         \hline
         \thead{Cut} & \multicolumn{2}{c}{\makecell{\vspace{-2.5mm}\\ $w \geq \left(\frac{1}{\lambda} - 1\right)\sum_{i \in [d_x]} q_i(\bar{x}_{2,i} - \bar{\sigma}_i)^2 + \bm{x}_2^\top \bm{Q}_2 \bm{x}_2 + \frac{\bm{y}\jf^\top \bm{R}\jf \bm{y}\jf + \bm{z}\kf^\top \bm{S}\kf \bm{z}\kf}{\lambda} + \frac{\bm{y}\js^\top \bm{R}\js \bm{y}\js + \bm{z}\ks^\top \bm{S}\ks \bm{z}\ks}{1-\lambda}$\vspace{0.8mm}}}\\[1pt] 
         \hline \hline
    \end{tabular}

%% file: tables/N=50_4models_comparison_arXiv.tex
\begin{tabular}{p{0.2cm}|p{0.2cm}|p{0.8cm}p{0.8cm}p{0.8cm}p{0.8cm}|p{0.8cm}p{0.8cm}p{0.8cm}p{0.8cm}}
\hline \hline
\multicolumn{2}{c|}{Dimensions}                & \multicolumn{4}{c|}{Relaxation Gap (\%)}                                                 & \multicolumn{4}{c}{Run Time (sec.)}                                                    \\ \hline
\multicolumn{1}{c|}{$d_x$}                 & \multicolumn{1}{c|}{$d_y$} & \multicolumn{1}{c}{MIQP} & \multicolumn{1}{c}{WC-NL} & \multicolumn{1}{c}{WC-G} & \multicolumn{1}{c|}{SPP} & \multicolumn{1}{c}{MIQP}   & \multicolumn{1}{c}{WC-NL}  & \multicolumn{1}{c}{WC-G}   & \multicolumn{1}{c}{SPP}    \\ \hline
\multicolumn{1}{c|}{\multirow{5}{*}{1}} & \multicolumn{1}{c|}{1}  & \multicolumn{1}{c}{60.5} & \multicolumn{1}{c}{45.4}  & \multicolumn{1}{c}{44.9} & \multicolumn{1}{c|}{\cellcolor[HTML]{FFFFC7}7.2} & \multicolumn{1}{c}{314.9}  & \multicolumn{1}{c}{1123.8} & \multicolumn{1}{c}{263.1}  & \multicolumn{1}{c}{\cellcolor[HTML]{DAE8FC}84.4}   \\ 
\multicolumn{1}{c|}{}                   & \multicolumn{1}{c|}{2}  & \multicolumn{1}{c}{71.9} & \multicolumn{1}{c}{57.7}  & \multicolumn{1}{c}{57.6} & \multicolumn{1}{c|}{\cellcolor[HTML]{FFFFC7}5.8} & \multicolumn{1}{c}{970.2}  & \multicolumn{1}{c}{1329.2} & \multicolumn{1}{c}{456.8}  & \multicolumn{1}{c}{\cellcolor[HTML]{DAE8FC}325.1}  \\ 
\multicolumn{1}{c|}{}                   & \multicolumn{1}{c|}{3}  & \multicolumn{1}{c}{73.9} & \multicolumn{1}{c}{60.0}  & \multicolumn{1}{c}{59.9} & \multicolumn{1}{c|}{\cellcolor[HTML]{FFFFC7}5.4} & \multicolumn{1}{c}{1922.5} & \multicolumn{1}{c}{918.5}  & \multicolumn{1}{c}{728.5}  & \multicolumn{1}{c}{\cellcolor[HTML]{DAE8FC}365.2}  \\ 
\multicolumn{1}{c|}{}                   & \multicolumn{1}{c|}{4}  & \multicolumn{1}{c}{74.4} & \multicolumn{1}{c}{63.1}  & \multicolumn{1}{c}{63.1} & \multicolumn{1}{c|}{\cellcolor[HTML]{FFFFC7}5.6} & \multicolumn{1}{c}{2035.2} & \multicolumn{1}{c}{839.8}  & \multicolumn{1}{c}{644.7}  & \multicolumn{1}{c}{\cellcolor[HTML]{DAE8FC}24.0}   \\ 
\multicolumn{1}{c|}{}                   & \multicolumn{1}{c|}{5}  & \multicolumn{1}{c}{75.9} & \multicolumn{1}{c}{64.9}  & \multicolumn{1}{c}{64.9} & \multicolumn{1}{c|}{\cellcolor[HTML]{FFFFC7}4.7} & \multicolumn{1}{c}{2768.6} & \multicolumn{1}{c}{1322.8} & \multicolumn{1}{c}{1399.0} & \multicolumn{1}{c}{\cellcolor[HTML]{DAE8FC}393.9}  \\ \hline
\multicolumn{1}{c|}{\multirow{5}{*}{2}} & \multicolumn{1}{c|}{2}  & \multicolumn{1}{c}{46.3} & \multicolumn{1}{c}{21.9}  & \multicolumn{1}{c}{21.8} & \multicolumn{1}{c|}{\cellcolor[HTML]{FFFFC7}0.3} & \multicolumn{1}{c}{1087.5} & \multicolumn{1}{c}{313.6}  & \multicolumn{1}{c}{8.0}    & \multicolumn{1}{c}{\cellcolor[HTML]{DAE8FC}2.1}    \\ 
\multicolumn{1}{c|}{}                   & \multicolumn{1}{c|}{3}  & \multicolumn{1}{c}{50.9} & \multicolumn{1}{c}{33.5}  & \multicolumn{1}{c}{33.3} & \multicolumn{1}{c|}{\cellcolor[HTML]{FFFFC7}0.2} & \multicolumn{1}{c}{1969.3} & \multicolumn{1}{c}{320.6}  & \multicolumn{1}{c}{330.9}  & \multicolumn{1}{c}{\cellcolor[HTML]{DAE8FC}2.4}    \\ 
\multicolumn{1}{c|}{}                   & \multicolumn{1}{c|}{4}  & \multicolumn{1}{c}{50.3} & \multicolumn{1}{c}{34.0}  & \multicolumn{1}{c}{34.0} & \multicolumn{1}{c|}{\cellcolor[HTML]{FFFFC7}0.1} & \multicolumn{1}{c}{1848.5} & \multicolumn{1}{c}{50.7}   & \multicolumn{1}{c}{135.1}  & \multicolumn{1}{c}{\cellcolor[HTML]{DAE8FC}1.1}    \\ 
\multicolumn{1}{c|}{}                   & \multicolumn{1}{c|}{5}  & \multicolumn{1}{c}{52.3} & \multicolumn{1}{c}{37.3}  & \multicolumn{1}{c}{37.3} & \multicolumn{1}{c|}{\cellcolor[HTML]{FFFFC7}0.0} & \multicolumn{1}{c}{2171.4} & \multicolumn{1}{c}{241.4}  & \multicolumn{1}{c}{191.0}  & \multicolumn{1}{c}{\cellcolor[HTML]{DAE8FC}0.9}    \\ 
\multicolumn{1}{c|}{}                   & \multicolumn{1}{c|}{6}  & \multicolumn{1}{c}{52.2} & \multicolumn{1}{c}{39.5}  & \multicolumn{1}{c}{39.5} & \multicolumn{1}{c|}{\cellcolor[HTML]{FFFFC7}0.0} & \multicolumn{1}{c}{1974.5} & \multicolumn{1}{c}{176.8}  & \multicolumn{1}{c}{390.9}  & \multicolumn{1}{c}{\cellcolor[HTML]{DAE8FC}0.7}    \\ \hline
\multicolumn{1}{c|}{\multirow{5}{*}{3}} & \multicolumn{1}{c|}{3}  & \multicolumn{1}{c}{24.9} & \multicolumn{1}{c}{9.8}   & \multicolumn{1}{c}{9.7}  & \multicolumn{1}{c|}{\cellcolor[HTML]{FFFFC7}0.5} & \multicolumn{1}{c}{2061.4} & \multicolumn{1}{c}{656.4}  & \multicolumn{1}{c}{37.1}   & \multicolumn{1}{c}{\cellcolor[HTML]{DAE8FC}7.4}    \\ 
\multicolumn{1}{c|}{}                   & \multicolumn{1}{c|}{4}  & \multicolumn{1}{c}{26.6} & \multicolumn{1}{c}{10.7}  & \multicolumn{1}{c}{10.7} & \multicolumn{1}{c|}{\cellcolor[HTML]{FFFFC7}0.6} & \multicolumn{1}{c}{3029.6} & \multicolumn{1}{c}{505.4}  & \multicolumn{1}{c}{662.3}  & \multicolumn{1}{c}{\cellcolor[HTML]{DAE8FC}11.8}   \\ 
\multicolumn{1}{c|}{}                   & \multicolumn{1}{c|}{5}  & \multicolumn{1}{c}{29.5} & \multicolumn{1}{c}{15.7}  & \multicolumn{1}{c}{15.7} & \multicolumn{1}{c|}{\cellcolor[HTML]{FFFFC7}0.3} & \multicolumn{1}{c}{2796.5} & \multicolumn{1}{c}{439.1}  & \multicolumn{1}{c}{627.3}  & \multicolumn{1}{c}{\cellcolor[HTML]{DAE8FC}4.3}    \\ 
\multicolumn{1}{c|}{}                   & \multicolumn{1}{c|}{6}  & \multicolumn{1}{c}{25.1} & \multicolumn{1}{c}{12.9}  & \multicolumn{1}{c}{12.9} & \multicolumn{1}{c|}{\cellcolor[HTML]{FFFFC7}0.6} & \multicolumn{1}{c}{2893.3} & \multicolumn{1}{c}{110.6}  & \multicolumn{1}{c}{439.9}  & \multicolumn{1}{c}{\cellcolor[HTML]{DAE8FC}68.6}   \\ 
\multicolumn{1}{c|}{}                   & \multicolumn{1}{c|}{7}  & \multicolumn{1}{c}{29.1} & \multicolumn{1}{c}{16.1}  & \multicolumn{1}{c}{16.1} & \multicolumn{1}{c|}{\cellcolor[HTML]{FFFFC7}0.1} & \multicolumn{1}{c}{3600.1} & \multicolumn{1}{c}{332.0}  & \multicolumn{1}{c}{876.6}  & \multicolumn{1}{c}{\cellcolor[HTML]{DAE8FC}2.4}    \\ \hline
\multicolumn{1}{c|}{\multirow{5}{*}{4}} & \multicolumn{1}{c|}{4}  & \multicolumn{1}{c}{7.6}  & \multicolumn{1}{c}{\cellcolor[HTML]{FFFFC7}0.8}   & \multicolumn{1}{c}{\cellcolor[HTML]{FFFFC7}0.8}  & \multicolumn{1}{c|}{4.5} & \multicolumn{1}{c}{3180.5} & \multicolumn{1}{c}{\cellcolor[HTML]{DAE8FC}631.4}  & \multicolumn{1}{c}{690.0}  & \multicolumn{1}{c}{2570.6} \\ 
\multicolumn{1}{c|}{}                   & \multicolumn{1}{c|}{5}  & \multicolumn{1}{c}{7.2}  & \multicolumn{1}{c}{\cellcolor[HTML]{FFFFC7}1.8}   & \multicolumn{1}{c}{\cellcolor[HTML]{FFFFC7}1.8}  & \multicolumn{1}{c|}{4.4} & \multicolumn{1}{c}{2669.9} & \multicolumn{1}{c}{463.5}  & \multicolumn{1}{c}{\cellcolor[HTML]{DAE8FC}208.9}  & \multicolumn{1}{c}{2408.4} \\ 
\multicolumn{1}{c|}{}                   & \multicolumn{1}{c|}{6}  & \multicolumn{1}{c}{6.8}  & \multicolumn{1}{c}{\cellcolor[HTML]{FFFFC7}1.6}   & \multicolumn{1}{c}{\cellcolor[HTML]{FFFFC7}1.6}  & \multicolumn{1}{c|}{4.7} & \multicolumn{1}{c}{2894.6} & \multicolumn{1}{c}{457.6}  & \multicolumn{1}{c}{\cellcolor[HTML]{DAE8FC}212.8}  & \multicolumn{1}{c}{3267.2} \\ 
\multicolumn{1}{c|}{}                   & \multicolumn{1}{c|}{7}  & \multicolumn{1}{c}{6.5}  & \multicolumn{1}{c}{\cellcolor[HTML]{FFFFC7}1.3}   & \multicolumn{1}{c}{\cellcolor[HTML]{FFFFC7}1.3}  & \multicolumn{1}{c|}{4.7} & \multicolumn{1}{c}{3240.4} & \multicolumn{1}{c}{\cellcolor[HTML]{DAE8FC}769.4}  & \multicolumn{1}{c}{843.5}  & \multicolumn{1}{c}{2870.0} \\ 
\multicolumn{1}{c|}{}                   & \multicolumn{1}{c|}{8}  & \multicolumn{1}{c}{5.2}  & \multicolumn{1}{c}{\cellcolor[HTML]{FFFFC7}1.3}   & \multicolumn{1}{c}{\cellcolor[HTML]{FFFFC7}1.3}  & \multicolumn{1}{c|}{5.2} & \multicolumn{1}{c}{2568.3} & \multicolumn{1}{c}{\cellcolor[HTML]{DAE8FC}656.1}  & \multicolumn{1}{c}{673.4}  & \multicolumn{1}{c}{3238.1} \\ \hline
\multicolumn{1}{c|}{\multirow{5}{*}{5}} & \multicolumn{1}{c|}{5}  & \multicolumn{1}{c}{3.3}  & \multicolumn{1}{c}{\cellcolor[HTML]{FFFFC7}0.4}   & \multicolumn{1}{c}{\cellcolor[HTML]{FFFFC7}0.4}  & \multicolumn{1}{c|}{3.5} & \multicolumn{1}{c}{1086.1} & \multicolumn{1}{c}{737.1}  & \multicolumn{1}{c}{\cellcolor[HTML]{DAE8FC}386.8}  & \multicolumn{1}{c}{1855.1} \\ 
\multicolumn{1}{c|}{}                   &\multicolumn{1}{c|}{6}  & \multicolumn{1}{c}{1.2}  & \multicolumn{1}{c}{\cellcolor[HTML]{FFFFC7}0.0}   & \multicolumn{1}{c}{\cellcolor[HTML]{FFFFC7}0.0}  & \multicolumn{1}{c|}{2.4} & \multicolumn{1}{c}{226.9}  & \multicolumn{1}{c}{21.6}   & \multicolumn{1}{c}{\cellcolor[HTML]{DAE8FC}0.2}    & \multicolumn{1}{c}{758.4}  \\ 
\multicolumn{1}{c|}{}                   & \multicolumn{1}{c|}{7}  & \multicolumn{1}{c}{1.3}  & \multicolumn{1}{c}{\cellcolor[HTML]{FFFFC7}0.1}   & \multicolumn{1}{c}{\cellcolor[HTML]{FFFFC7}0.1}  & \multicolumn{1}{c|}{2.6} & \multicolumn{1}{c}{366.8}  & \multicolumn{1}{c}{\cellcolor[HTML]{DAE8FC}29.1}   & \multicolumn{1}{c}{51.6}   & \multicolumn{1}{c}{852.8}  \\ 
\multicolumn{1}{c|}{}                   & \multicolumn{1}{c|}{8}  & \multicolumn{1}{c}{0.6}  & \multicolumn{1}{c}{\cellcolor[HTML]{FFFFC7}0.0}   & \multicolumn{1}{c}{\cellcolor[HTML]{FFFFC7}0.0}  & \multicolumn{1}{c|}{2.6} & \multicolumn{1}{c}{4.8}    & \multicolumn{1}{c}{20.4}   & \multicolumn{1}{c}{\cellcolor[HTML]{DAE8FC}0.3}    & \multicolumn{1}{c}{492.1}  \\ 
\multicolumn{1}{c|}{}                   & \multicolumn{1}{c|}{9}  & \multicolumn{1}{c}{0.5}  & \multicolumn{1}{c}{\cellcolor[HTML]{FFFFC7}0.0}   & \multicolumn{1}{c}{\cellcolor[HTML]{FFFFC7}0.0}  & \multicolumn{1}{c|}{2.9} & \multicolumn{1}{c}{82.7}   & \multicolumn{1}{c}{20.8}   & \multicolumn{1}{c}{\cellcolor[HTML]{DAE8FC}0.4}    & \multicolumn{1}{c}{1458.4} \\ \hline
\multicolumn{2}{c|}{Average}                 & \multicolumn{1}{c}{\textbf{31.4}} & \multicolumn{1}{c}{\textbf{21.2}}  & \multicolumn{1}{c}{\textbf{21.1}} & \multicolumn{1}{c|}{\cellcolor[HTML]{FFFFC7}\textbf{2.8}} & \multicolumn{1}{c}{\textbf{1910.6}} & \multicolumn{1}{c}{\textbf{499.5}}  & \multicolumn{1}{c}{\cellcolor[HTML]{DAE8FC}\textbf{410.4}}  & \multicolumn{1}{c}{\textbf{842.6}}  \\ \hline \hline
\end{tabular}

%% file: tables/WC_SPP_dimensions_arXiv.tex
\begin{tabular}{c|c|c}
     \hline \hline
     Model& WC-G & SPP\\ \hline 
     State $x$ & $d_x (n+1)$ & $d_x ((d_z+1)n+2)$ \\ \hline
     Control $y$ & $d_y n$ & $d_y ((d_z+1)n+2)$ \\ \hline
     Indicator $z$  & $d_z n$ & $(d_z+1)^2 (n-1) + 2(d_z+1)$
     \\ \hline \hline
\end{tabular}

%% file: tables/HEV_result_arXiv.tex
\begin{tabular}{p{0.15cm}|p{0.2cm}|p{0.2cm}|p{0.7cm}p{0.7cm}p{0.7cm}p{0.7cm}p{0.7cm}p{0.7cm}|p{0.7cm}p{0.7cm}p{0.7cm}p{0.7cm}p{0.7cm}p{0.7cm}}
\hline \hline
\multicolumn{3}{c|}{Sampling Time}  & \multicolumn{6}{c|}{$T_s=1$ sec}                                                                                                                         & \multicolumn{6}{c}{$T_s=0.5$ sec}                                                                                                                       \\ \hline
\multicolumn{3}{c|}{Instance}  & \multicolumn{2}{c|}{Relax. Gap (\%)}                         & \multicolumn{2}{c|}{Comp. Time (sec)}                       & \multicolumn{2}{c|}{\# Branch}   & \multicolumn{2}{c|}{Relax. Gap (\%)}                         & \multicolumn{2}{c|}{Comp. Time (sec)}                       & \multicolumn{2}{c}{\# Branch}   \\ \hline
\multicolumn{1}{c|}{$r_1$} & \multicolumn{1}{c|}{$\gamma$} & \multicolumn{1}{c|}{Metrics} & \multicolumn{1}{c}{MIQP}  & \multicolumn{1}{c|}{WC-G}   & \multicolumn{1}{c}{MIQP} & \multicolumn{1}{c|}{WC-G}   & \multicolumn{1}{c}{MIQP} & \multicolumn{1}{c|}{WC-G}   & \multicolumn{1}{c}{MIQP}  & \multicolumn{1}{c|}{WC-G}   & \multicolumn{1}{c}{MIQP} & \multicolumn{1}{c|}{WC-G}   & \multicolumn{1}{c}{MIQP} & \multicolumn{1}{c}{WC-G}   \\ \hline
\multicolumn{1}{c|}{\multirow{6}{*}{$1$}} & \multicolumn{1}{c|}{\multirow{2}{*}{$0$}} & \multicolumn{1}{c|}{Avg.}                                                                & \multicolumn{1}{c}{5.09}  & \multicolumn{1}{c|}{0.00} & \multicolumn{1}{c}{0.03} & \multicolumn{1}{c|}{0.01} & \multicolumn{1}{c}{2.49} & \multicolumn{1}{c|}{0} & \multicolumn{1}{c}{3.64}  & \multicolumn{1}{c|}{0.00} & \multicolumn{1}{c}{0.03} & \multicolumn{1}{c|}{0.01} & \multicolumn{1}{c}{1.7} & \multicolumn{1}{c}{0} \\ 
& &\multicolumn{1}{c|}{Max}                                                                 & \multicolumn{1}{c}{90.85} & \multicolumn{1}{c|}{0.00} & \multicolumn{1}{c}{0.13} & \multicolumn{1}{c|}{0.03} & \multicolumn{1}{c}{41}   & \multicolumn{1}{c|}{0}    & \multicolumn{1}{c}{98.20} & \multicolumn{1}{c|}{0.00} & \multicolumn{1}{c}{0.15} & \multicolumn{1}{c|}{0.02} & \multicolumn{1}{c}{41}   & \multicolumn{1}{c}{0}    \\ \cline{2-15} 
 & \multicolumn{1}{c|}{\multirow{2}{*}{$0.01$}} & \multicolumn{1}{c|}{Avg.}                                                                & \multicolumn{1}{c}{5.09}  & \multicolumn{1}{c|}{0.05} & \multicolumn{1}{c}{0.04} & \multicolumn{1}{c|}{0.02} & \multicolumn{1}{c}{3.94} & \multicolumn{1}{c|}{1.19} & \multicolumn{1}{c}{3.64}  & \multicolumn{1}{c|}{0.48}  & \multicolumn{1}{c}{0.03} & \multicolumn{1}{c|}{0.02} & \multicolumn{1}{c}{1.66} & \multicolumn{1}{c}{1.18} \\ 
& & \multicolumn{1}{c|}{Max}                                                       & \multicolumn{1}{c}{90.85} & \multicolumn{1}{c|}{3.25} & \multicolumn{1}{c}{0.70} & \multicolumn{1}{c|}{0.21} & \multicolumn{1}{c}{139}  & \multicolumn{1}{c|}{39}   & \multicolumn{1}{c}{98.24} & \multicolumn{1}{c|}{38.99} & \multicolumn{1}{c}{0.32} & \multicolumn{1}{c|}{0.19} & \multicolumn{1}{c}{41}   & \multicolumn{1}{c}{39}   \\ \cline{2-15}
 & \multicolumn{1}{c|}{\multirow{2}{*}{$0.02$}} & \multicolumn{1}{c|}{Avg.}                                                                & \multicolumn{1}{c}{5.09}  & \multicolumn{1}{c|}{0.14} & \multicolumn{1}{c}{0.04} & \multicolumn{1}{c|}{0.02} & \multicolumn{1}{c}{3.12} & \multicolumn{1}{c|}{1.64} & \multicolumn{1}{c}{5.97}  & \multicolumn{1}{c|}{0.93}  & \multicolumn{1}{c}{0.05} & \multicolumn{1}{c|}{0.02} & \multicolumn{1}{c}{41.04} & \multicolumn{1}{c}{3.12} \\ 
& & \multicolumn{1}{c|}{Max}                                                       & \multicolumn{1}{c}{90.85} & \multicolumn{1}{c|}{7.59} & \multicolumn{1}{c}{0.64} & \multicolumn{1}{c|}{0.28} & \multicolumn{1}{c}{101}  & \multicolumn{1}{c|}{45}   & \multicolumn{1}{c}{96.77} & \multicolumn{1}{c|}{37.73} & \multicolumn{1}{c}{1.77} & \multicolumn{1}{c|}{0.27} & \multicolumn{1}{c}{3913}  & \multicolumn{1}{c}{209}  \\ \hline
\multicolumn{1}{c|}{\multirow{6}{*}{$10$}} & \multicolumn{1}{c|}{\multirow{2}{*}{$0$}} & \multicolumn{1}{c|}{Avg.}                                                                & \multicolumn{1}{c}{4.58}  & \multicolumn{1}{c|}{0.00} & \multicolumn{1}{c}{0.04} & \multicolumn{1}{c|}{0.02} & \multicolumn{1}{c}{2.48} & \multicolumn{1}{c|}{0.06} & \multicolumn{1}{c}{3.47}  & \multicolumn{1}{c|}{0.00} & \multicolumn{1}{c}{0.04} & \multicolumn{1}{c|}{0.02} & \multicolumn{1}{c}{1.69} & \multicolumn{1}{c}{0.04} \\ 
& & \multicolumn{1}{c|}{Max}          &                                                       \multicolumn{1}{c}{87.02} & \multicolumn{1}{c|}{0.00} & \multicolumn{1}{c}{0.42} & \multicolumn{1}{c|}{0.14} & \multicolumn{1}{c}{41}   & \multicolumn{1}{c|}{5}    & \multicolumn{1}{c}{97.90} & \multicolumn{1}{c|}{0.00} & \multicolumn{1}{c}{0.60} & \multicolumn{1}{c|}{0.12} & \multicolumn{1}{c}{41}   & \multicolumn{1}{c}{3}    \\ \cline{2-15}
 & \multicolumn{1}{c|}{\multirow{2}{*}{$0.01$}} & \multicolumn{1}{c|}{Avg.}                                                                & \multicolumn{1}{c}{4.58}  & \multicolumn{1}{c|}{0.03} & \multicolumn{1}{c}{0.29}  & \multicolumn{1}{c|}{0.06} & \multicolumn{1}{c}{62.97} & \multicolumn{1}{c|}{5.82} & \multicolumn{1}{c}{3.47}  & \multicolumn{1}{c|}{0.44}  & \multicolumn{1}{c}{2.53}   & \multicolumn{1}{c|}{0.31}  & \multicolumn{1}{c}{721.76} & \multicolumn{1}{c}{57.58} \\ 
& & \multicolumn{1}{c|}{Max}                                                                 & \multicolumn{1}{c}{87.01} & \multicolumn{1}{c|}{2.88} & \multicolumn{1}{c}{23.01} & \multicolumn{1}{c|}{3.63} & \multicolumn{1}{c}{5697}  & \multicolumn{1}{c|}{573}  & \multicolumn{1}{c}{97.91} & \multicolumn{1}{c|}{33.43} & \multicolumn{1}{c}{198.37} & \multicolumn{1}{c|}{27.93} & \multicolumn{1}{c}{63583}  & \multicolumn{1}{c}{6017}  \\ \cline{2-15}
 & \multicolumn{1}{c|}{\multirow{2}{*}{$0.02$}} & \multicolumn{1}{c|}{Avg.}                                                                & \multicolumn{1}{c}{4.58}  & \multicolumn{1}{c|}{0.09} & \multicolumn{1}{c}{0.21}  & \multicolumn{1}{c|}{0.10} & \multicolumn{1}{c}{40.44} & \multicolumn{1}{c|}{12.64} & \multicolumn{1}{c}{5.79}   & \multicolumn{1}{c|}{0.94}  & \multicolumn{1}{c}{1.52}   & \multicolumn{1}{c|}{0.14} & \multicolumn{1}{c}{431.49} & \multicolumn{1}{c}{27.87} \\ 
& & \multicolumn{1}{c|}{Max}                                                                 & \multicolumn{1}{c}{87.01} & \multicolumn{1}{c|}{5.85} & \multicolumn{1}{c}{12.74} & \multicolumn{1}{c|}{5.18} & \multicolumn{1}{c}{3131}  & \multicolumn{1}{c|}{895}   & \multicolumn{1}{c}{100.00} & \multicolumn{1}{c|}{39.51} & \multicolumn{1}{c}{107.14} & \multicolumn{1}{c|}{8.29} & \multicolumn{1}{c}{30791}  & \multicolumn{1}{c}{2407}  \\ \hline
\multicolumn{1}{c|}{\multirow{6}{*}{$100$}} & \multicolumn{1}{c|}{\multirow{2}{*}{$0$}} & \multicolumn{1}{c|}{Avg.}                                                                & \multicolumn{1}{c}{4.48}  & \multicolumn{1}{c|}{0.00} & \multicolumn{1}{c}{0.05} & \multicolumn{1}{c|}{0.02} & \multicolumn{1}{c}{2.54} & \multicolumn{1}{c|}{0.01} & \multicolumn{1}{c}{3.43}  & \multicolumn{1}{c|}{0.00} & \multicolumn{1}{c}{0.04} & \multicolumn{1}{c|}{0.02} & \multicolumn{1}{c}{1.66} & \multicolumn{1}{c}{0.01} \\ 
& & \multicolumn{1}{c|}{Max}                                                                 & \multicolumn{1}{c}{85.78} & \multicolumn{1}{c|}{0.00} & \multicolumn{1}{c}{0.28} & \multicolumn{1}{c|}{0.13} & \multicolumn{1}{c}{41}   & \multicolumn{1}{c|}{1}    & \multicolumn{1}{c}{97.86} & \multicolumn{1}{c|}{0.00} & \multicolumn{1}{c}{0.29} & \multicolumn{1}{c|}{0.12} & \multicolumn{1}{c}{41}   & \multicolumn{1}{c}{1}    \\ \cline{2-15}
 & \multicolumn{1}{c|}{\multirow{2}{*}{$0.01$}} & \multicolumn{1}{c|}{Avg.}                                                                & \multicolumn{1}{c}{4.48}  & \multicolumn{1}{c|}{0.03} & \multicolumn{1}{c}{1.69}   & \multicolumn{1}{c|}{0.08} & \multicolumn{1}{c}{368.3} & \multicolumn{1}{c|}{9.7} & \multicolumn{1}{c}{3.66}  & \multicolumn{1}{c|}{0.43}  & \multicolumn{1}{c}{7.03}   & \multicolumn{1}{c|}{0.75}  & \multicolumn{1}{c}{1688.5} & \multicolumn{1}{c}{130.1} \\ 
& & \multicolumn{1}{c|}{Max}                                                                 & \multicolumn{1}{c}{85.78} & \multicolumn{1}{c|}{2.98} & \multicolumn{1}{c}{135.16} & \multicolumn{1}{c|}{6.08} & \multicolumn{1}{c}{30635}  & \multicolumn{1}{c|}{963}  & \multicolumn{1}{c}{97.88} & \multicolumn{1}{c|}{32.88} & \multicolumn{1}{c}{381.19} & \multicolumn{1}{c|}{46.38} & \multicolumn{1}{c}{93660}   & \multicolumn{1}{c}{8585}   \\ \cline{2-15}
 & \multicolumn{1}{c|}{\multirow{2}{*}{$0.02$}} & \multicolumn{1}{c|}{Avg.}                                                                & \multicolumn{1}{c}{4.77}  & \multicolumn{1}{c|}{0.09} & \multicolumn{1}{c}{4.69}   & \multicolumn{1}{c|}{0.37}  & \multicolumn{1}{c}{1174.2} & \multicolumn{1}{c|}{62.6} & \multicolumn{1}{c}{4.97}  & \multicolumn{1}{c|}{0.84}  & \multicolumn{1}{c}{7.36}   & \multicolumn{1}{c|}{0.38}  & \multicolumn{1}{c}{1666.3} & \multicolumn{1}{c}{58.2} \\ 
& & \multicolumn{1}{c|}{Max}                                                                 & \multicolumn{1}{c}{85.77} & \multicolumn{1}{c|}{6.06} & \multicolumn{1}{c}{286.14} & \multicolumn{1}{c|}{23.20} & \multicolumn{1}{c}{74288}   & \multicolumn{1}{c|}{4467}  & \multicolumn{1}{c}{97.61} & \multicolumn{1}{c|}{45.98} & \multicolumn{1}{c}{367.27} & \multicolumn{1}{c|}{28.89} & \multicolumn{1}{c}{86356}   & \multicolumn{1}{c}{4915}   \\ \hline 
\multicolumn{3}{c|}{\textbf{Overall Average}} & \multicolumn{1}{c}{\textbf{4.75}}	&\multicolumn{1}{c|}{\textbf{0.05}}	&\multicolumn{1}{c}{\textbf{0.79}}	&\multicolumn{1}{c|}{\textbf{0.08}}	&\multicolumn{1}{c}{\textbf{184.50}}	&\multicolumn{1}{c|}{\textbf{10.41}} & \multicolumn{1}{c}{\textbf{4.23}}	&\multicolumn{1}{c|}{\textbf{0.45}}	&\multicolumn{1}{c}{\textbf{2.0}7}	&\multicolumn{1}{c|}{\textbf{0.19}}	&\multicolumn{1}{c}{\textbf{506.19}}	&\multicolumn{1}{c}{\textbf{30.9}}\\ \hline \hline
\end{tabular}